\documentclass[psamsfonts,12pt]{amsart}
\usepackage{graphicx}  
\usepackage{amsmath,amssymb,amsthm, amscd}
\usepackage{amssymb,fge}
\usepackage{tikz-cd}

\usepackage{color}

\DeclareMathOperator{\Orb}{Orb}

\DeclareMathOperator{\PrePer}{PrePer}

\DeclareMathOperator{\Per}{Per}

\newtheorem{theorem}{Theorem}[section]
\newtheorem{lemma}[theorem]{Lemma}

\newtheorem{proposition-definition}[theorem]{Proposition-Definition}
\newtheorem{corollary}[theorem]{Corollary}
\newtheorem{conjecture}[theorem]{Conjecture}

\theoremstyle{definition}

\theoremstyle{remark}
\newtheorem{remark}[theorem]{Remark}

\usepackage[margin=1in]{geometry}  
\usepackage{graphicx}              
\usepackage{amsmath}               
\usepackage{amsfonts}              
\usepackage{amsthm}                
\usepackage{verbatim}              
\usepackage{indentfirst}           
\usepackage{underscore}
\usepackage{enumitem}
\usepackage{relsize}
\usepackage{varwidth}
\usepackage{mathtools} 
\usepackage{hyperref}
\hypersetup{
    colorlinks=true,
    linkcolor=blue,
    filecolor=magenta,      
    urlcolor=cyan,
}
\usepackage{amssymb}
\usepackage[all]{xy}
\makeatletter
\renewcommand*\env@matrix[1][*\c@MaxMatrixCols c]{%
  \hskip -\arraycolsep
  \let\@ifnextchar\new@ifnextchar
  \array{#1}}
\makeatother

\theoremstyle{remark}

\newcommand*{\Scale}[2][4]{\scalebox{#1}{$#2$}}%

\newcommand{\Mod}[1]{\ (\textup{mod}\ #1)}     

\title[Finite orbit points for sets of quadratic polynomials]{Finite orbit points for sets of quadratic polynomials} 
\begin{document}
\author[Wade Hindes]{Wade Hindes}
\address{Department of Mathematics, Texas State University, 601 University Dr., San Marcos, TX 78666.}
\email{wmh33@txstate.edu}
\maketitle
\renewcommand{\thefootnote}{}
\begin{abstract} Let $S=\{x^2+c_1, x^2+c_2,\dots, x^2+c_s\}$ be a set of quadratic polynomials with rational coefficients, and let $P$ be a rational basepoint. We classify the pairs $(S,P)$ for which $P$ has finite orbit for $S$, assuming that the maximum period length for each individual polynomial is at most three (conjectured by Poonen). In particular, under these hypotheses we prove that if $s\geq4$, then there are no points $P$ with finite orbit for $S$. Moreover, we use this perspective to formulate an analog of the Morton-Silverman Conjecture for sets of rational functions.       
\end{abstract} 
\section{Introduction}
Let $S$ be a set of rational functions with coefficients in a number field $K$, and let $M_S$ be the monoid of all functions obtained by composing finitely many elements of $S$: 
\[M_S=\Big\{\rho\in K(x)\,: \rho=(\phi_n\circ\phi_{n-1}\circ\dots \circ\phi_1)\;\; \text{for $\phi_i\in S$, $n\in\mathbb{N}$}\Big\}.\]
Then given a basepoint $P\in\mathbb{P}^1(K)$ we define the \emph{orbit} of $P$ with respect to $S$ to be, 
\[\Orb_S(P):=\big\{\rho(P)\in\mathbb{P}^1(K)\;:\:\text{for $\rho\in M_S$}\big\}.\] 
As with the case when $S$ contains a single function, where we can view the monoid $M_S$ as inducing a discrete dynamical system on $\mathbb{P}^1(K)$, we may study the set of \emph{finite orbit points for $S$}, i.e., the set of points $P\in\mathbb{P}^1(K)$ such that $\Orb_S(P)$ is a finite set. In particular, it follows from Northcott's theorem (applied to a single function) that the set of finite orbit points for $S$ (over a fixed $K$) is finite whenever $S$ contains a function of degree at least two; see \cite[Theorem 3.12]{SilvDyn}. On the other hand, if the maps in $S$ are sufficiently ``dynamically independent", then one expects that there can be no points with finite orbit for $S$. We test this heuristic when $S$ is a set of quadratic polynomials and $K=\mathbb{Q}$; see Theorem \ref{thm:classification} below. 
 
However, before we can discuss our results we need some additional notation. Given a rational map $\phi\in K(x)$ and an integer $n\geq1$, let $\Per_n^{**}(\phi,K)$ denote the set of points in $\mathbb{P}^1(K)$ of \emph{exact period} $n$ for $\phi$. That is, $\Per_n^{**}(\phi,K)$ denotes the set of points $P$ such that $\phi^n(P)=P$ and $\phi^m(P)\neq P$ for all $m<n$; see \cite[\S4.1]{SilvDyn}. In particular, it follows again from Northcott's theorem that $\Per_n^{**}(\phi,K)$ is empty for all $n$ sufficiently large; here we assume that $\deg(\phi)\geq2$. With this in mind, we define the quantity   
\[\mu_S(K)=\sup\big\{n:\,\Per_n^{**}(\phi,K)\neq \varnothing\;\;\text{for some $\phi\in S$}\big\},\] 
to be the maximum exact $K$-period length over all maps in $S$. In particular, the main result of this paper states that if $S$ is a set of quadratic polynomials over the field $K=\mathbb{Q}$, then there are no points $P\in\mathbb{Q}$ with finite orbit for $S$, as long as $\mu_S(\mathbb{Q})$ is not too large:         
\begin{theorem}{\label{thm:classification}} Let $S=\{x^2+c_1, x^2+c_2, \dots ,x^2+c_s\}$ for some distinct $c_i\in\mathbb{Q}$, and suppose that $\mu_S(\mathbb{Q})\leq3$. Then the following statements hold:   
\begin{enumerate}[topsep=8pt, partopsep=5pt, itemsep=7pt]  
\item[\textup{(1)}] If $\#S=3$ and $P\in\mathbb{Q}$ has finite orbit for $S$, then \vspace{.2cm}
\[S=\bigg\{x^2-\frac{5}{16},\, x^2-\frac{13}{16},\, x^2-\frac{21}{16}\bigg\}\;\;\text{and}\;\;P\in\bigg\{\pm\frac{1}{4},\,\pm\frac{3}{4},\, \pm\frac{5}{4}\bigg\} \]
or  \vspace{.1cm}
\[\hspace{-.77cm}S=\bigg\{x^2+\frac{3}{16},\, x^2-\frac{5}{16},\, x^2-\frac{13}{16}\bigg\}\;\;\text{and}\;\;P\in\bigg\{\pm\frac{1}{4},\,\pm\frac{3}{4}\bigg\}. \vspace{.2cm}\]
\item[\textup{(2)}] If $\#S\geq4$, then there are no points $P\in\mathbb{Q}$ with finite orbit for $S$.    
\end{enumerate} 
\end{theorem} 
As an application, we note that Theorem \ref{thm:classification} holds without assumptions on $\mu_S(\mathbb{Q})$ when the polynomials in $S$ have integral coefficients. To see this, combine \cite[Exercise 2.20]{SilvDyn} and \cite[Theorem 4]{4cycle} to deduce that $\mu_S(\mathbb{Q})\leq2$ in this case.   
\begin{corollary}{\label{cor:int}} Let $S=\{x^2+c_1, x^2+c_2, \dots ,x^2+c_s\}$ for some distinct $c_i\in\mathbb{Z}$. If $\#S\geq3$, then there are no points $P\in\mathbb{Q}$ with finite orbit for $S$. 
\end{corollary} 
\begin{remark} Note that Corollary \ref{cor:int} is sharp. For instance, if $y\in\mathbb{Z}$ satisfies $y\equiv{\pm{1}}\Mod{4}$, then the pair $(S,P)$ given by 
\[S=\Big\{x^2+\frac{1-y^2}{4},\; x^2+\frac{-3-y^2}{4}\Big\}\;\;\text{and}\;\;\; P=\frac{1+y}{2}\]
has finite orbit; see Lemma \ref{lem:12} below.  
\end{remark}
On the other hand, the assumption that $\mu_S(\mathbb{Q})\leq3$ is believed to hold in general. Specifically, Poonen has conjectured that the maximum rational period length of a  quadratic polynomial with rational coefficients is three:   
\begin{conjecture}[Poonen \cite{Poonen}] There are no $c\in\mathbb{Q}$ such that $\phi(x)=x^2+c$ has a $\mathbb{Q}$-rational periodic point with exact period greater than three.  
\end{conjecture} 
\begin{remark} Some evidence for Poonen's conjecture is as follows. It has been verified for all $c\in\mathbb{Q}$ with height at most $10^8$; see \cite{Hutz-Ingram}. Moreover, it is known that there are no rational $c$ values with rational periodic points of exact period $4$ and $5$; see \cite{4cycle} and \cite{5cycle} respectively. Likewise, assuming the BSD conjecture for a single curve of genus $4$, there are no rational points of exact period $6$; see \cite{6cycle}. In particular, Theorem \ref{thm:classification} remains true unconditionally with the hypothesis $\mu_S(\mathbb{Q})\leq5$.    
\end{remark}
In particular, Poonen's Conjecture and Theorem \ref{thm:classification} together imply that there are no rational points with finite orbit whenever $S$ contains at least four quadratic polynomials: 
\begin{conjecture}{\label{conj1}} Let $S=\{x^2+c_1, x^2+c_2, \dots ,x^2+c_s\}$ for some distinct $c_i\in\mathbb{Q}$. If $\#S\geq4$, then there are no points $P\in\mathbb{Q}$ with finite orbit for $S$. 
\end{conjecture} 
Moreover, as with Poonen's conjecture and preperiodic points in general, we can rephrase the results and conjectures above in terms of height functions. Namely, for a fixed map $\phi$, it is well known that there is a canonical height function $\hat{h}_\phi:\mathbb{P}^1\rightarrow\mathbb{R}$ whose kernel is exactly the points of finite orbit for $\phi$; see \cite[Theorem 3.22]{SilvDyn}. Similarly, if $S$ is a height controlled collection of maps and $\nu$ is any positive probability measure on $S$ (see \cite[\S1.1]{VOHWH}), then the author et. al. showed that there exists a height function $\mathbb{E}_\nu[\hat{h}]:\mathbb{P}^1\rightarrow \mathbb{R}$ whose kernel is exactly the set of finite orbit points for $S$ (as defined above); see \cite[Corollary 1.4]{VOHWH}. In particular, such a height function exists \cite{Kawaguchi} for any finite collection of maps $S$ and any positive probability measure $\nu$ on $S$; here positive means that $\nu(\phi)>0$ for all $\phi\in S$. Hence, we can restate Theorem \ref{thm:classification} in terms of heights as follows: 
\begin{theorem}\label{thm:hts} Let $S=\{x^2+c_1, x^2+c_2, \dots ,x^2+c_s\}$ for some distinct $c_i\in\mathbb{Q}$ and let $\nu$ be any positive probability measure on $S$. If $\mu_S(\mathbb{Q})\leq3$ and $\#S\geq4$, then
\[\big\{P\in\mathbb{P}^1(\mathbb{Q})\,: \mathbb{E}_\nu[\hat{h}](P)=0\big\}=\{\infty\}.\]
Equivalently, there are no affine rational points of expected height zero for $S$. 
\end{theorem} 
Moreover, we can similarly characterize Corollary \ref{cor:int} and Conjecture \ref{conj1} in terms of heights: replace ``points $P$ with finite orbit for $S$" with ``points $P$ with $\mathbb{E}_\nu[\hat{h}](P)=0$." Finally, with Theorem \ref{thm:classification}, Conjecture \ref{conj1}, and \cite[Conjecture 3.15]{SilvDyn} in mind, we make the following conjecture, which may be viewed as an analog of the Morton-Silverman Conjecture \cite{Morton-Silverman} (concerning finite orbit points for a single map) for sets of maps. 
\begin{conjecture}\label{conj:me} Let $S=\{\phi_1, \phi_2, \dots,\phi_s\}$ be a set of rational maps defined over a number field $K$ such that $2\leq\deg(\phi)\leq d$ for all $\phi\in S$. Then there exist constants $B=B(K,d)$ and $C=C(K,d)$, depending only on $K$ and $d$, such that if $\#S\geq C$, then there are at most $B$ points $P\in \mathbb{P}^1(K)$ of finite orbit for $S$. Equivalently, if $\nu$ is any positive probability measure on $S$ and $\#S\geq C$, then 
 \begin{equation}\label{bdexphtzero}
 \#\big\{P\in\mathbb{P}^1(K)\,: \mathbb{E}_\nu[\hat{h}](P)=0\big\}\leq B.
 \end{equation} 
That is, there are at most $B$ points of expected height zero for $S$. 
\end{conjecture}
\begin{remark} It is worth pointing out that Conjecture \ref{conj:me} follows from \cite[Conjecture 3.15]{SilvDyn}: simply use for $B$ the uniform bound on the number of preperiodic points for a single map in $S$ provided by \cite[Conjecture 3.15]{SilvDyn} and set $C=1$. However, in the author's opinion, Conjecture \ref{conj:me} is more tractable - there are more (at the very least, different) tools at one's disposal when one is allowed to add maps to the monoid $M_S$; see the proofs of the classification results for two maps with finite orbit points in Section \ref{sec:twomaps}.   
\end{remark} 
On the other hand, it follows from \cite[Corollary 1.3]{Baker-DeMarco} that there are finitely many algebraic points of finite orbit for $S$ whenever the maps in $S$ have distinct Julia sets, subject of course to some embedding of the ground field $K$ into the complex numbers; see \cite{SilvDyn} for the relevant background material on Julia sets. In particular, it would be interesting to know if a version of Conjecture \ref{conj:me}, specifically the bound in (\ref{bdexphtzero}), holds for the number of $\bar{\mathbb{Q}}$-finite orbit points for $S$, assuming that the maps in $S$ have distinct Julia sets and bounded degree, and that the cardinality of $S$ is sufficiently large. However, we have not thought about this problem enough to warrant any formal statement or conjecture.

Finally, for those readers interested in additional work on arithmetic aspects of monoid dynamics, see \cite{monoid1, monoid2, monoid3}. We now proceed with the proof of Theorem \ref{thm:classification}. In particular, we rely heavily upon the computer algebra system {\tt{Magma}} throughout. Our computations and codes can be found at: \vspace{.05cm}  
\[\boxed{{\tt{https://sites.google.com/a/alumni.brown.edu/whindes/research}}}\vspace{.05cm}\]  
\textbf{Acknowledgements:} It is a pleasure to thank Vivian Olsiewski Healey, Trevor Hyde, Giacomo Micheli, and Joseph Silverman for discussions related to the work in this paper. 
\section{Pairs of maps with finite orbit points}\label{sec:twomaps} 
We begin with several classification results for pairs of quadratic polynomials $S=\{f_1,f_2\}$ with a point of finite orbit. The main idea throughout is to compute a Gr\"{o}bner basis for the ideals in $\mathbb{Q}[c_1,c_2]$ generated by several polynomials of the form $f_1^{n_1}(f_2^{n_2}(P))-f_1^{n_3}(f_2^{n_2}(P))$; here the iterates $n_i$ are chosen according to the classification results in \cite{Poonen}. As we will see, $P$ is (usually) a polynomial in $\mathbb{Q}[c_1,c_2]$ also chosen according to the main results in \cite{Poonen}. 
\begin{remark} For the {\tt{Magma}} codes verifying the calculations in this section, see the file named \emph{Two maps with points of finite orbit} at the website mentioned in the introduction.  
\end{remark} 
We begin with a classification result for sets of quadratic polynomials $S=\{f_1,f_2\}$ possessing a specified finite orbit point when both $f_1$ and $f_2$ have rational fixed points.      
\begin{lemma}{\label{lem:11}} Let $S=\{f_1, f_2\}$ with $f_1=x^2+c_1$ and $f_2=x^2+c_2$ for some distinct $c_1,c_2\in\mathbb{Q}$. If $S$ has the following properties:    
\begin{enumerate}[topsep=7pt, partopsep=7pt, itemsep=8pt]
\item[\textup{(1)}] $\mu_S(\mathbb{Q})\leq3$,  
\item[\textup{(2)}] $f_1$ and $f_2$ both have rational fixed points,
\item[\textup{(3)}] a fixed point $P\in\mathbb{Q}$ of $f_1$ has finite orbit for $S$, 
\end{enumerate}  
then at least one of the following statements holds:  \begin{enumerate}[topsep=7pt, partopsep=7pt, itemsep=8pt]
\item[\textup{(a)}] $(c_1,c_2,P)=\Big(\frac{1-y^2}{4},\,\frac{1-(y+2)^2}{4},\, \frac{1+y}{2}\Big)$ for some $y\in\mathbb{Q}$,  
\item[\textup{(b)}] $(c_1,c_2,P)=\Big(\frac{t^4 - 18t^2 + 1}{4(t^2 -1)^2},\, \frac{-3t^4 - 10t^2 -3}{4(t^2 -1)^2},\, \frac{t^2 + 4t - 1}{2(t^2 - 1)}\Big)$ for some $t\in\mathbb{Q}$,
\item[\textup{(c)}] $(c_1,c_2)\in\Big\{\big(-\frac{21}{16},-\frac{5}{16}\big), \big(\frac{3}{16},\frac{-5}{16}\big)\Big\}$. 
\end{enumerate}   
\end{lemma}
\begin{proof} We note first that if $f=x^2+c$ for some $c\in\mathbb{Q}$ and $f$ has a rational fixed-point $P\in\mathbb{Q}$, then 
$(c,P)=(\frac{1-y^2}{4},\frac{1+y}{2})$ for some $y\in\mathbb{Q}$. This follows easily from an examination of the quadratic equation $0=f(x)-x=x^2-x+c$ and determining when this equation has a rational root $P$, i.e., when the discriminant of $x^2-x+c$ is a rational square. In particular, we may assume that: 
\[c_1=\frac{1-y^2}{4},\; P=\frac{1+y}{2},\; c_2=\frac{1-z^2}{4}\;\;\; \text{for $y,z\in\mathbb{Q}$}.\]
We need not keep track of the fixed point of $f_2$, since we make no assumptions about the action of $S$ on it. On the other hand, since $P$ is assumed to have finite orbit for $S$, it must be the case that $P$ is preperiodic for $f_2$ and $Q:=f_2^2(P)$ is preperiodic for $f_1$. In particular, since $\mu_S(\mathbb{Q})\leq3$ and both $f_1$ and $f_2$ have rational fixed points, Theorems 1, 2, and 3 of \cite{Poonen} together imply that 
\begin{equation}{\label{idealgen11}}
f_2^4(P)-f_2^2(P)=0 \;\;\text{and}\;\; f_1^4(Q)-f_1^2(Q)=0;
\end{equation} 
the key point here is that the classification theorems in \cite{Poonen} imply that all rational preperiodic points (of maps with rational fixed points) enter a fixed point or $2$-cycle after two iterates, assuming $\mu_S(\mathbb{Q})\leq3$. 

We now view the rational numbers $y,z\in\mathbb{Q}$ as specializations of the variables \textbf{y} and \textbf{z}, and let $F_1(\textbf{y},\textbf{z})=f_{2,\textbf{z}}^4(P_{\textbf{y}})-f_{2,\textbf{z}}^2(P_{\textbf{y}})$ and $F_2(\textbf{y},\textbf{z})=f_{1,\textbf{y}}^4(Q_{\textbf{y},\textbf{z}})-f_{1,\textbf{y}}^2(Q_{\textbf{y},\textbf{z}})$ be polynomials in $\mathbb{Q}[\textbf{y},\textbf{z}]$; here we use subscripts to emphasize on what variables the various functions depend (in the proof of Lemma \ref{lem:11} only). Now, we compute with {\tt{Magma}} that the polynomial  
\[F(\textbf{y},\textbf{z}):=G_{\textbf{z}}(\textbf{y}-\textbf{z})(\textbf{y} + \textbf{z})(\textbf{y} - \textbf{z} + 2)(\textbf{y} + \textbf{z} + 2)(\textbf{y}^2 - \textbf{z}^2 + 4)(\textbf{y}^2 + 4\textbf{y} - \textbf{z}^2 + 8)\]
is in the $\mathbb{Q}[\textbf{y},\textbf{z}]$-ideal generated by $F_1(\textbf{y},\textbf{z})$ and $F_2(\textbf{y},\textbf{z})$ (by computing a Gr\"{o}bner basis for this ideal); here $G_{\textbf{z}}\in\mathbb{Q}[\textbf{z}]$ is a polynomial in the variable $\textbf{z}$ alone of degree $28$. In particular, it follows from (\ref{idealgen11}) that a rational pair $(y,z)\in\mathbb{Q}^2$ associated to a finite orbit tuple $(c_1,c_2,P)$ as in Lemma \ref{lem:11} must satisfy $F(y,z)=0$. On the other hand, we compute with {\tt{Magama}} that the only rational roots of $G_{\textbf{z}}$ are $z=$$\pm{1},\pm{2},\pm{3/2}$. Therefore, such a pair $(y,z)\in\mathbb{Q}^2$ must satisfy:  
\[z\in\Big\{ \pm{1},\pm{2},\pm{3/2}\Big\},\;\;\; z=\pm{y},\;\;z=\pm(y+2),\;\;\; y^2-z^2=-4,\;\;\text{or}\;\;\; y^2+4y-z^2=-8.\vspace{.15cm}\] 
However, note that if $z=\pm y$, then $c_1=c_2$, an immediate contradiction. We proceed to evaluate the other possibilities. Suppose first that $y=\pm{(y+2)}$. Then, in either case, the tuple $(c_1,c_2,P)$ takes the form: 
\begin{equation}{\label{finiteorbitpair1}} 
(c_1,c_2,P)=\Big(\frac{1-y^2}{4},\,\frac{1-(y+2)^2}{4},\, \frac{1+y}{2}\Big)\;\;\;\text{for $y\in\mathbb{Q}$}.
\end{equation}  
Moreover, it is straightforward to check that the set $\{P,-P, f_2(P), f_1(f_2(P))\}$ is stable under the action of $S$. In particular, $(c_1,c_2,P)$ is a finite orbit tuple in this case. 

Suppose next that $y^2-z^2=-4$. Since, $y^2-z^2=-4$ is a rational curve with a rational point $(0,2)$, we may parametrize its rational points via 
\[(y,z)=\Big(\frac{4t}{t^2 - 1},\frac{2t^2 + 2}{t^2 - 1}\Big)\;\;\; \text{for $t\in\mathbb{Q}$}.\] 
In particular, with this parametrization the tuple $(c_1,c_2,P)$ takes the form:\vspace{.1cm}  
\[ (c_1,c_2,P)=\bigg(\frac{t^4 - 18t^2 + 1}{4(t^2 -1)^2},\, \frac{-3t^4 - 10t^2 -3}{4(t^2 -1)^2},\, \frac{t^2 + 4t - 1}{2(t^2 - 1)}\bigg)\;\;\; \text{for $t\in\mathbb{Q}$}.\vspace{.1cm}\]             
Moreover, we compute with {\tt{Magma}} that the finite set $\big\{P,-P, f_2(P),f_1(f_2(P))\big\}$ is stable under the action of $S$. In particular, $(c_1,c_2,P)$ is a finite orbit tuple in this case.  

Suppose next that $y^2+4y-z^2=-8$. Since, $y^2+4y-z^2=-8$ is a rational curve with a rational point $(-2,2)$, we may parametrize its rational points via 
\[(y,z)=\Big(\frac{-2t^2 + 4t + 2}{t^2 - 1},\frac{2t^2 + 2}{t^2 - 1}\Big)\;\;\; \text{for $t\in\mathbb{Q}$}.\] 
In particular, with this parametrization the tuple $(c_1,c_2,P)$ takes the form: \vspace{.1cm}  
\[ (c_1,c_2,P)=\bigg(\frac{-3t^4 + 16t^3 - 10t^2 - 16t - 3}{4(t^2 -1)^2},\, \frac{-3t^4 - 10t^2 -3}{4(t^2 -1)^2},\, \frac{-t^2 + 4t + 1}{2(t^2 - 1)}\bigg)\;\;\; \text{for $t\in\mathbb{Q}$}.\vspace{.1cm}\]   
However, since $P$ has finite orbit for $S$, so does $W:=f_2(P)$. In particular, $W$ must be preperiodic for $f_1$. Therefore, Theorems 1, 2, and 3 of \cite{Poonen} together imply that $f_1^4(W)-f_1^2(W)=0$. However, this expression is in $t$ alone, and we check with {\tt{Magma}} that $t=0$ is the only rational solution of this equation. Therefore, we may assume that $t=0$. However, in this case $c_1=-3/4=c_2$, a contradiction. 

Finally, suppose that $z\in\big\{ \pm{1},\pm{2},\pm{3/2}\big\}$ and let $I$ be the $\mathbb{Q}[\textbf{y},\textbf{z}]$-ideal generated by $F_1(\textbf{y},\textbf{z})$ and $F_2(\textbf{y},\textbf{z})$. We compute a Gr\"{o}bner basis $\{B_1(\textbf{y},\textbf{z}),B_2(\textbf{y},\textbf{z}),B_3(\textbf{y},\textbf{z})\}$ for $I$ with {\tt{Magma}}; in fact, the function $F(\textbf{y},\textbf{z})$ above is $B_3(\textbf{y},\textbf{z})$. In particular, since every element of $I$ must vanish at $(z,y)$, a pair of rational numbers associated to a finite orbit tuple tuple $(c_1,c_2,P)$ as in Lemma \ref{lem:11}, the polynomials $B_1(\textbf{y},z)\in\mathbb{Q}[\textbf{y}]$ for $z\in\big\{\pm{1},\pm{2},\pm{3/2}\big\}$ must vanish at their corresponding $y$ value. However, we compute with {\tt{Magma}} the possible rational roots of these six polynomials and deduce that 
\[z\in\big\{ \pm{1},\pm{2},\pm{3/2}\big\}\;\; \text{implies} \;\; y\in\big\{0,\pm{1},\pm{2},\pm{1/2},\pm{3/2},-3,-4,-7/2,-5/2\big\}.\] 
On the other hand, $B_2(\textbf{y},\textbf{z})$ and $B_3(\textbf{y},\textbf{z})$ must also vanish at the pairs above, and among the pairs where this is so, all but those associated to $(c_1,c_2)=\big(-\frac{21}{16},-\frac{5}{16}\big)$ and $(c_1,c_2)=\big(\frac{3}{16},-\frac{5}{16}\big)$ are accounted for in the finite orbit family, $(c_1,c_2,P)=\big(\frac{1+y^2}{4},\frac{1-(y+2)^2}{4}, \frac{1+y}{2}\big)$, analyzed in (\ref{finiteorbitpair1}) above. This completes the proof of Lemma \ref{lem:11}.        
\end{proof} 
We next prove a classification result for sets of quadratic polynomials $S=\{f_1,f_2\}$ possessing a specified finite orbit point when $f_1$ has a rational fixed point and $f_2$ has a rational point of period two.
\begin{lemma}{\label{lem:12}} Let $S=\{f_1, f_2\}$ with $f_1=x^2+c_1$ and $f_2=x^2+c_2$ for some distinct $c_1,c_2\in\mathbb{Q}$. If $S$ has the following properties:    
\begin{enumerate}[topsep=7pt, partopsep=7pt, itemsep=8pt]
\item[\textup{(1)}] $\mu_S(\mathbb{Q})\leq3$,  
\item[\textup{(2)}] $f_1$ has a rational fixed point and $f_2$ has a rational point of period two,  
\item[\textup{(3)}] a fixed point $P\in\mathbb{Q}$ of $f_1$ has finite orbit for $S$, 
\end{enumerate}  
then at least one of the following statements holds:  \begin{enumerate}[topsep=7pt, partopsep=7pt, itemsep=8pt]
\item[\textup{(a)}] $(c_1,c_2,P)=\Big(\frac{1-y^2}{4},\,\frac{-3-y^2}{4},\, \frac{1+y}{2}\Big)$ for some $y\in\mathbb{Q}$,  
\item[\textup{(b)}] $(c_1,c_2,P)=\Big(\frac{-15t^4 - 2t^2 + 1}{4(t^2 -1)^2},\, \frac{-3t^4 - 10t^2 -3}{4(t^2 -1)^2},\, \frac{-3t^2- 1}{2(t^2 - 1)}\Big)$ for some $t\in\mathbb{Q}$,
\item[\textup{(c)}] $(c_1,c_2)\in\Big\{\big(-\frac{5}{16},-\frac{13}{16}\big),\;\big(-\frac{21}{16},-\frac{13}{16}\big) \Big\}$. 
\end{enumerate}   
\end{lemma}
\begin{proof} We note first that if $f=x^2+c$ for some $c\in\mathbb{Q}$ and $f$ has a rational fixed-point $P\in\mathbb{Q}$, then 
$(c,P)=(\frac{1-y^2}{4},\frac{1+y}{2})$ for some $y\in\mathbb{Q}$. Likewise, if $f=x^2+c$ has a point $W\in\mathbb{Q}$ of exact period two, then 
$(c,W)=\big(\frac{-(3+z^2)}{4},\frac{-1+z}{2}\big)$ for some $z\in\mathbb{Q}$. These characterizations follow easily from examining the equations $x^2+c-x=0$ and $(x^2+c-x)(x^2 + x + c + 1)=(x^2+c)^2+c-x=0$ respectively. In particular, since all the polynomials involved are quadratic in $x$, they have rational roots if and only if their discriminants are rational squares. Hence, we may assume that: 
\[c_1=\frac{1-y^2}{4},\; P=\frac{1+y}{2},\; c_2=\frac{-(3+z^2)}{4}\;\; \text{for some $y,z\in\mathbb{Q}$};\]
we need not keep track of the rational point of period two for $f_2$, since we make no assumptions about the action of $S$ on it. On the other hand, since $P$ is assumed to have finite orbit for $S$, it must be the case that $P$ is preperiodic for $f_2$ and $Q:=f_2^2(P)$ is preperiodic for $f_1$. In particular, since $\mu_S(\mathbb{Q})\leq3$, $f_1$ has a rational fixed point, and $f_2$ has a rational point of period two, Theorems 1, 2, and 3 of \cite{Poonen} together imply that 
\begin{equation}{\label{idealgen12}}
f_2^4(P)-f_2^2(P)=0 \;\;\text{and}\;\; f_1^4(Q)-f_1^2(Q)=0;
\end{equation} 
the key point here is that the classification theorems in \cite{Poonen} imply that all rational preperiodic points (of maps with fixed points or points of period two) enter a fixed point or $2$-cycle after two iterates, assuming as always, $\mu_S(\mathbb{Q})\leq3$.

We now view the rational numbers $y,z\in\mathbb{Q}$ as specializations of the variables \textbf{y} and \textbf{z}, and let $F_1(\textbf{y},\textbf{z})=f_{2}^4(P)-f_2^2(P)$ and $F_2(\textbf{y},\textbf{z})=f_{1}^4(Q)-f_{1}^2(Q)$ be polynomials in $\mathbb{Q}[\textbf{y},\textbf{z}]$. Now, we compute with {\tt{Magma}} that the polynomial \vspace{.1cm}  
\[F(\textbf{y},\textbf{z}):=G_{\textbf{z}}(\textbf{y}-\textbf{z})(\textbf{y} + \textbf{z})(\textbf{y} - \textbf{z} + 2)(\textbf{y} + \textbf{z} + 2)(\textbf{y}^2 - \textbf{z}^2 - 4)(\textbf{y}^2 + 4\textbf{y} - \textbf{z}^2) \vspace{.1cm}\]
is in the $\mathbb{Q}[\textbf{y},\textbf{z}]$-ideal generated by $F_1(\textbf{y},\textbf{z})$ and $F_2(\textbf{y},\textbf{z})$ (by computing a Gr\"{o}bner basis for this ideal); here $G_{\textbf{z}}\in\mathbb{Q}[\textbf{z}]$ is a polynomial in the variable $\textbf{z}$ of degree $30$. In particular, it follows from (\ref{idealgen12}) that the rational pair $(y,z)\in\mathbb{Q}^2$, associated to a finite orbit tuple $(c_1,c_2,P)$ as in Lemma \ref{lem:12}, must satisfy $F(y,z)=0$. On the other hand, {\tt{Magama}} computes that the only rational roots of $G_{\textbf{z}}$ are $z=0$ and $z=\pm{1/2}$. Therefore, $(y,z)\in\mathbb{Q}^2$ parametrizing $(c_1,c_2,P)$ must satisfy: 
\[z\in\Big\{0,\pm{1/2}\Big\},\;\;\; z=\pm{y},\;\;z=\pm(y+2),\;\;\; y^2-z^2=4,\;\;\text{or}\;\;\; y^2+4y-z^2=0.\]       
We proceed to evaluate these possibilities. Suppose first that $y^2-z^2=4$. Then since $y^2-z^2=4$ is a rational curve with a rational point $(2,0)$, we may parametrize its rational points via 
\[(y,z)=\Big(\frac{-2t^2-2}{t^2 - 1},\frac{-4t}{t^2 - 1}\Big)\;\;\; \text{for $t\in\mathbb{Q}$}.\] 
In particular, with this parametrization for $(y,z)$ the tuple $(c_1,c_2,P)$ takes the form: 
\[ (c_1,c_2,P)=\bigg(\frac{-3t^4 - 10t^2-3}{4(t^2 -1)^2},\, \frac{-3t^4 - 10t^2 -3}{4(t^2 -1)^2},\, \frac{-t^2-3}{2(t^2 - 1)}\bigg)\;\;\; \text{for $t\in\mathbb{Q}$}.\] 
Therefore, $c_1=c_2$ in this case, and we obtain an immediate contradiction.  

Suppose next that $y^2+4y-z^2=0$. Then since $y^2+4y-z^2=0$ is a rational curve with a rational point $(0,0)$, we may parametrize its rational points via 
\[(y,z)=\Big(\frac{-4t^2}{t^2 - 1},\frac{-4t}{t^2 - 1}\Big)\;\;\; \text{for $t\in\mathbb{Q}$}.\] 
In particular, with this parametrization for $(y,z)$ the tuple $(c_1,c_2,P)$ takes the form: 
\[ (c_1,c_2,P)=\bigg(\frac{-15t^4 - 2t^2 + 1}{4(t^2 -1)^2},\, \frac{-3t^4 - 10t^2 -3}{4(t^2 -1)^2},\, \frac{-3t^2- 1}{2(t^2 - 1)}\bigg)\;\;\; \text{for $t\in\mathbb{Q}$}.\]
Moreover, we check with {\tt{Magma}} that the set $\{P,-P\}$ is stable under the action of $S$ in this case. Hence, $P$ has finite orbit for $S$. 

Suppose next that $z=\pm{y}$. Then the tuple $(c_1,c_2,P)$ takes the form: 
\begin{equation}{\label{finiteorbitpair2}} 
(c_1,c_2,P)=\Big(\frac{1-y^2}{4},\,\frac{-3-y^2}{4},\, \frac{1+y}{2}\Big)\;\;\; \text{for some $y\in\mathbb{Q}$}.
\end{equation} 
Moreover, we check with {\tt{Magma}} that the set $\{P,-P,f_2(P), f_1(f_2(P))\}$ is stable under the action of $S$ in this case. Hence, $P$ has finite orbit for $S$. 

Suppose next that $z=\pm{(y+2)}$. Then the tuple $(c_1,c_2,P)$ takes the form: 
\[(c_1,c_2,P)=\Big(\frac{1-y^2}{4},\,\frac{-3-(y+2)^2}{4},\, \frac{1+y}{2}\Big)\;\;\; \text{for some $y\in\mathbb{Q}$}.\] 
However, since $P$ has finite orbit for $S$, so does $W:=(f_2\circ f_1\circ f_2\circ f_1\circ f_2\circ f_2)(P)$. In particular, $W$ must be preperiodic for $f_1$. Therefore, Theorems 1, 2, and 3 of \cite{Poonen} together imply that $f_1^4(W)-f_1^2(W)=0$. However, this expression is in $y$ alone, and we check with {\tt{Magma}} that $y=-1,-3/2,-2,-5/2$ are the only rational solutions to this equation. However, if $y=-1,-3/2, -5/2$, then we obtain the tuples \[(c_1,c_2,P)=(0,-1,0),\; (-5/16,-13/16, -1/4),\; (-21/16,-13/16,-3/4)\]
with finite orbit. On the other hand, the tuple $(c_1,c_2,P)=(0,-1,0)$ is already accounted for in the family in (\ref{finiteorbitpair2}) above by setting $y=-1$.  

Finally, suppose that $z\in\{0,\pm{1/2}\}$ and let $I$ be the $\mathbb{Q}[\textbf{y},\textbf{z}]$-ideal generated by $F_1(\textbf{y},\textbf{z})$ and $F_2(\textbf{y},\textbf{z})$. We compute a Gr\"{o}bner basis $\{B_1(\textbf{y},\textbf{z}),B_2(\textbf{y},\textbf{z}),B_3(\textbf{y},\textbf{z})\}$ for $I$ with {\tt{Magma}}; in fact the function $F(\textbf{y},\textbf{z})$ above is $B_3(\textbf{y},\textbf{z})$. In particular, since every element of $I$ must vanish at $(y,z)$, a pair of rational numbers associated to a finite orbit tuple $(c_1,c_2,P)$ as in Lemma \ref{lem:12}, the polynomials $B_2(\textbf{y},z)\in\mathbb{Q}[\textbf{y}]$ for $z\in\big\{0,\pm{1/2}\big\}$ must have a rational root at their corresponding $y$ value. However, we compute with {\tt{Magma}} the possible rational roots of these polynomials, and deduce that 
\[z\in\big\{0,\pm{1/2}\big\}\;\; \text{implies} \;\; y\in\big\{\pm{1/2},\pm{3/2},-5/2\big\}.\] 
On the other hand, $B_2(\textbf{y},\textbf{z})$ and $B_3(\textbf{y},\textbf{z})$ must also vanish at these pairs, and among the pairs where this is so, all but those associated to $(c_1,c_2)=\big(-\frac{5}{16},-\frac{13}{16}\big)$ and $(c_1,c_2)=\big(-\frac{21}{16},-\frac{13}{16}\big)$ are accounted for in the family (\ref{finiteorbitpair2}) above. 
\end{proof} 
We next prove a classification result for sets of quadratic polynomials $S=\{f_1,f_2\}$ possessing a specified finite orbit point when both $f_1$ and $f_2$ have rational points of period two. 
\begin{lemma}{\label{lem:22}} Let $S=\{f_1, f_2\}$ with $f_1=x^2+c_1$ and $f_2=x^2+c_2$ for some distinct $c_1,c_2\in\mathbb{Q}$. If $S$ has the following properties:    
\begin{enumerate}[topsep=7pt, partopsep=7pt, itemsep=8pt]
\item[\textup{(1)}] $\mu_S(\mathbb{Q})\leq3$,  
\item[\textup{(2)}] $f_1$ and $f_2$ have rational points of period two,  
\item[\textup{(3)}] a point $P\in\mathbb{Q}$ of period two for $f_1$ has finite orbit for $S$, 
\end{enumerate}  
then at least one of the following statements holds:  \begin{enumerate}[topsep=7pt, partopsep=7pt, itemsep=8pt] 
\item[\textup{(a)}] $(c_1,c_2,P)=\Big(\frac{-7t^4 - 2t^2 -7}{4(t^2 -1)^2},\, \frac{-3t^4 - 10t^2 -3}{4(t^2 -1)^2},\, \frac{-3t^2- 1}{2(t^2 - 1)}\Big)$ for some $t\in\mathbb{Q}$,
\item[\textup{(b)}] $(c_1,c_2)\in \Big\{\big(-\frac{3}{4},-\frac{7}{4}\big),\big(-\frac{7}{4},-\frac{3}{4}\big)\;\big(-\frac{13}{16},-\frac{21}{16}\big),\big(-\frac{21}{16},-\frac{13}{16}\big) \; \big(-\frac{37}{16},-\frac{21}{16}\big) \Big\}$. 
\end{enumerate}   
\end{lemma} 
\begin{proof} Since $f_1$ and $f_2$ have rational points of period two, we may assume that the finite orbit tuple $(c_1,c_2,P)$ in Lemma \ref{lem:22} takes the form 
\[(c_1,c_2,P)=\bigg(\frac{-(3+y^2)}{4}, \frac{-(3+z^2)}{4}, \frac{-1+y}{2}\bigg) \;\;\; \text{for some $y,z\in\mathbb{Q}$}.\] 
On the other, since $P$ has finite orbit for $S$, it must be the case that $P$, $Q_1:=f_2(P)$ and $Q_2:=f_1(f_2(f_1(P)))$ are preperiodic for both $f_1$ and $f_2$. Hence, Theorems 1, 2, and 3 of \cite{Poonen} imply that 
\[f_2^4(P)-f_2^2(P)=0, \;\; f_1^4(Q_1)-f_1^2(Q_1)=0,\;\;\text{and}\;\; f_2^4(Q_2)-f_2^2(Q_2)=0.\]  
With these relations in mind, let $N(\textbf{y},\textbf{z})$ be the polynomial in $\mathbb{Q}[\textbf{y},\textbf{z}]$ corresponding to $f_2^4(P)-f_2^2(P)$, let $A_1(\textbf{y},\textbf{z})$ be the polynomial corresponding to $f_1^4(Q_1)-f_1^2(Q_1)$, and let $A_2(\textbf{y},\textbf{z})$ be the polynomial corresponding to $f_2^4(Q_2)-f_2^2(Q_2)$. In particular, if $I$ denotes the $\mathbb{Q}[\textbf{y},\textbf{z}]$ ideal generated by $N$, $A_1$ and $A_2$, then the elements of $I$ must vanish at any $(y,z)\in\mathbb{Q}^2$ associated to a finite orbit tuple $(c_1,c_2,P)$ as in Lemma \ref{lem:22}. However, we compute a Gr\"{o}bner basis $\{B_1(\textbf{y},\textbf{z}),B_2(\textbf{y},\textbf{z}), B_3(\textbf{y},\textbf{z})\}$ for $I$ with {\tt{Magma}} and find that $B_3(\textbf{y},\textbf{z})$ has a factorization  
\[B_3(\textbf{y},\textbf{z})=G(\textbf{z})(\textbf{y}-\textbf{z})(\textbf{y}+\textbf{z})(\textbf{y}^2-\textbf{z}^2-4),\] 
for some $G(\textbf{z})\in\mathbb{Q}[\textbf{z}]$ of degree 64. Moreover, we compute with {\tt{Magma}} that the only rational roots of $G(\textbf{z})$ are $z\in\{0,\pm{1},\pm{2},\pm{1/2},\pm{3/2}\}$. Therefore, if $(y,z)$ corresponds to a finite orbit tuple $(c_1,c_2,P)$ in Lemma \ref{lem:22}, then 
\[z\in\{0,\pm{1},\pm{2},\pm{1/2},\pm{3/2}\}, \;\ y=\pm{z},\;\;\text{or}\;\; y^2-z^2=4.\] 
However, we note that if $y=\pm{z}$, then $c_1=c_2$, and we obtain an immediate contradiction in this case. Therefore, we proceed to examine the remaining cases now. 

Suppose first that $y^2-z^2=4$. Then since $y^2-z^2=4$ is a rational curve with point $(2,0)$, we may parametrize its rational solutions via:   
\[(y,z)=\Big(\frac{-4t^2}{t^2 - 1},\frac{-4t}{t^2 - 1}\Big)\;\;\; \text{for $t\in\mathbb{Q}$}.\] 
In particular, the tuple $(c_1,c_2,P)$ takes the form
\[(c_1,c_2,P)=\bigg(\frac{-7t^4 - 2t^2 -7}{4(t^2 -1)^2},\, \frac{-3t^4 - 10t^2 -3}{4(t^2 -1)^2},\, \frac{-3t^2- 1}{2(t^2 - 1)}\bigg)\;\;\; \text{for $t\in\mathbb{Q}$}.\] 
Moreover, we check with {\tt{Magma}} that the set $\big\{P, f_2(P), f_1(f_2(P)), -f_1(f_2(P))\big\}$ is stable under the action of $S$. In particular, $(c_1,c_2,P)$ is a finite orbit pair in this case. 

Finally, suppose that $z\in\{0,\pm{1},\pm{2},\pm{1/2},\pm{3/2}\}$. Then the polynomials $B_i(\textbf{y},z)\in\mathbb{Q}[\textbf{y}]$ for $z\in\{0,\pm{1},\pm{2},\pm{1/2},\pm{3/2}\}$ and $1\leq i\leq3$ must rational roots at their corresponding $y$ value. However, we compute with {\tt{Magma}} the possible rational roots of these polynomials and deduce that 
\[z\in\{0,\pm{1},\pm{2},\pm{1/2},\pm{3/2}\}\;\; \text{implies} \;\; y\in\big\{0,\pm{1},\pm{2},\pm{1/2},\pm{3/2},\pm{5/2}\big\}.\]  
Moreover, among these values, we obtain the pairs of maps given by
\[(c_1,c_2)\in\Scale[.99]{\bigg\{\Big(-\frac{3}{4},-\frac{7}{4}\Big),\Big(-\frac{7}{4},-\frac{3}{4}\Big)\;\Big(-\frac{13}{16},-\frac{21}{16}\Big),\Big(-\frac{21}{16},-\frac{13}{16}\Big) \; \Big(-\frac{37}{16},-\frac{21}{16}\Big) \bigg\}},\] all of which posses finite orbit rational points; see our {\tt{Magma}} code for more details.  
\end{proof} 
We next prove a classification result for sets of quadratic polynomials $S=\{f_1,f_2\}$ possessing an unspecified finite orbit point when both $f_1$ and $f_2$ have a rational point of period three. 
\begin{lemma}{\label{lem:33}} Let $S=\{f_1, f_2\}$ with $f_1=x^2+c_1$ and $f_2=x^2+c_2$ for some distinct $c_1,c_2\in\mathbb{Q}$. If $S$ has the following properties:     
\begin{enumerate}[topsep=7pt, partopsep=7pt, itemsep=8pt]
\item[\textup{(1)}] $\mu_S(\mathbb{Q})\leq3$,  
\item[\textup{(2)}] $f_1$ and $f_2$ have rational points of period three,   
\end{enumerate}  
then there are no points $P\in\mathbb{Q}$ with finite orbit for $S$.   
\end{lemma} 
\begin{proof} We begin with a few preliminaries. First, since $f_1\neq f_2$, we may assume that at least one $c_i\neq-29/16$. Without loss of generality, we assume that $c_1\neq-29/16$. Next, note that if $f=x^2+c$ for some $c\in\mathbb{Q}$ and $f$ has a rational point of period three, then the tuple $(c,P_1,P_2,P_3)$ takes the form \vspace{.15cm}    
\begin{equation}{\label{3cycle}}
\begin{split}
c=&\frac{-(y^6+2y^5+4y^4+8y^3+9y^2+4y+1)}{4y^2(y+1)^2},\\[5pt]
P_1=\frac{y^3-y-1}{2y(y+1)}&,\;\; P_2=\frac{y^3+2y^2+y+1}{(2y(y+1)},\;\; P_3=\frac{-(y^3+2y^2+3y+1)}{2y(y+1)}
\end{split} 
\end{equation} 
for some $y\in\mathbb{Q}$; here $P_1, P_2$ and $P_3$ are the unique points of period three for $f$. For a justification of this fact, see \cite[Theorem 1]{Poonen}. In particular, both $f_1$ and $f_2$ have parametrizations of this form. Moreover, we use $y\in\mathbb{Q}$ to parametrize $(c_1,P_1,P_2,P_3)$ and $t\in\mathbb{Q}$ to parametrize $(c_2,Q_1,Q_2,Q_3)$ respectively. Now suppose that there exists a point $P\in\mathbb{Q}$ with finite orbit for $S$. In particular, $P$ must be preperiodic for $f_2$. Hence, Theorems 1, 2, and 3 of \cite{Poonen} together imply that the $f_2$-orbit of $P$ must contain $Q_1$ (all preperiodic orbits contain the unique $3$-cycle); here we use that $\mu_S(\mathbb{Q})\leq3$. In particular, the $S$-orbit of $P$ must contain the $S$-orbit of $Q_1$. Therefore, we may assume that $Q_1$ has finite orbit for $S$.

On the other hand, if $Q_1$ has finite orbit for $S$, then $Q_1$ is preperiodic for $f_1$. In particular, since $c_1\neq-29/16$, it must be the case that $f_1(Q_1)\in\{P_1,P_2,P_3\}$; see \cite[Theorem 3]{Poonen}. Similarly, $f_2(Q_1)$ must be preperiodic for $f_1$, and hence $f_1(f_2(Q_1))\in\{P_1,P_2,P_3\}$. Now let $N_i(\textbf{y},\textbf{t})$ be the polynomial in two variables given by the numerator of the rational expression $f_1(Q_1)-P_i$; here we view $y,t\in\mathbb{Q}$ as specializations of the variables $\textbf{y}$ and $\textbf{t}$ respectively. Likewise, let $A_i(\textbf{y},\textbf{t})$ be the numerator of the expression $f_1(f_2(Q_1)-P_i$. In particular, setting $N(\textbf{y},\textbf{t})=N_1N_2N_3$ and $A(\textbf{y},\textbf{t})=A_1A_2A_3$, we see that 
\[N(y,t)=0\;\;\; \text{and}\;\;\; A(y,t)=0,\] 
for all pairs $(y,t)$ associated to finite orbit tuples $(c_1,c_2,P)$ as in Lemma \ref{lem:33}. On the other hand, we compute with {\tt{Magma}} that the $\mathbb{Q}[\textbf{y},\textbf{t}]$ ideal generated by $N(\textbf{y},\textbf{t})$ and $A(\textbf{y},\textbf{t})$ contains the polynomial 
\[G(\textbf{t})(\textbf{y}-\textbf{t})(\textbf{yt} + \textbf{t} + 1)(\textbf{yt} + \textbf{y} + 1);\]
here $G(\textbf{t})\in\mathbb{Q}[\textbf{t}]$ has degree $38$. In particular, if $(y,t)\in\mathbb{Q}^2$ is associated to finite orbit tuples $(c_1,c_2,P)$ as in Lemma \ref{lem:33}, then $G(t)=0$, $y=t$, $y=-(t+1)/t$, or $y=-1/(t+1)$. However, the only rational roots of $G(\textbf{t})$ are $t=0,-1$, and both of these choices for $t$ force $c_2$ to be infinite, a contradiction. Moreover, if $y=t$, $y=-(t+1)/t$, or $y=-1/(t+1)$, then we check with {\tt{Magma}} that $c_1=c_2$, a contradiction. Hence, there cannot exist a point $P\in\mathbb{Q}$ with finite orbit for $S$ as claimed.               
\end{proof} 
We next prove a classification result for sets of quadratic polynomials $S=\{f_1,f_2\}$ possessing an unspecified finite orbit point when $f_1$ has a rational fixed point and $f_2$ has a rational point of period three.
\begin{lemma}{\label{lem:13}} Let $S=\{f_1, f_2\}$ with $f_1=x^2+c_1$ and $f_2=x^2+c_2$ for some distinct $c_1,c_2\in\mathbb{Q}$. If $S$ has the following properties:     
\begin{enumerate}[topsep=7pt, partopsep=7pt, itemsep=8pt]
\item[\textup{(1)}] $\mu_S(\mathbb{Q})\leq3$,  
\item[\textup{(2)}] $f_1$ has a rational fixed point and $f_2$ has a rational point of period three, 
\item[\textup{(3)}] there exists $P\in\mathbb{Q}$ with finite orbit for $S$,    
\end{enumerate}  
then $c_1=-\frac{21}{16}$ and $c_2=-\frac{29}{16}$.     
\end{lemma}
\begin{proof} Since $f_1$ has a rational fixed point, we may write $c_1=(1-y^2)/4$ for some $y\in\mathbb{Q}$. Likewise, since $f_2$ has has a rational point of period three, we may parametrize $f_2$ and the $3$-cycle $(c_2,P_1,P_2,P_3)$ as in (\ref{3cycle}); however, we use the variable $t$ to parametrize $f_2$ and the $3$-cycle in this case. Now suppose that $P\in\mathbb{Q}$ has finite orbit for $S$. Then $P$ is preperiodic for $f_2$. In particular, it follows from Theorems 1, 2, and 3 of \cite{Poonen} that the $f_2$-orbit of $P$ contains the $3$-cycle $P_1$,$P_2$ and $P_3$; here we use that $\mu_S(\mathbb{Q})\leq3$. Hence, the $S$-orbit of $P$ contains $P_1$. Therefore, we may assume that $P=P_1$ has finite orbit for $S$. 

Now, since $P_1$ has finite orbit for $S$, it follows that $P_1$ is preperiodic for $f_1$. Hence, Theorems 1, 2, and 3 of \cite{Poonen} imply that $f_1^4(P_1)=f_1^2(P_1)$; here we use that $f_1$ has a fixed point and that $\mu_S(\mathbb{Q})\leq3$. Likewise, $Q:=f_2(P_1)$ is preperiodic for $f_1$ and $f_1^4(Q)=f_1^2(Q)$. Now let $N(\textbf{y},\textbf{t})$ be the polynomial in two variables given by the numerator of the rational expression $f_1^4(P_1)-f_1^2(P_1)$, and let $A(\textbf{y},\textbf{t})$ be the numerator of the expression $f_1^4(Q)-f_1^2(Q)$. In particular, we have that 
\[N(y,t)=0\;\;\; \text{and}\;\;\; A(y,t)=0,\] 
for all pairs $(y,t)$ associated to finite orbit tuples $(c_1,c_2,P)$ as in Lemma \ref{lem:13}. On the other hand, we compute a Gr\"{o}bner basis $\{B_1(\textbf{y},\textbf{t}),B_2(\textbf{y},\textbf{t}),\dots B_6(\textbf{y},\textbf{t})\}$ for the $\mathbb{Q}[\textbf{y},\textbf{t}]$ ideal $I$ generated by $N(\textbf{y},\textbf{t})$ and $A(\textbf{y},\textbf{t})$. In particular, we see that $I$ contains a polynomial $G(\textbf{t})\in\mathbb{Q}[\textbf{t}]$ (in $\textbf{t}$ alone) of degree $68$. In particular, if $(y,t)$ is associated to a finite orbit tuple $(c_1,c_2,P)$ as in Lemma \ref{lem:13}, then $G(t)=0$. However, we compute with {\tt{Magma}} that the only rational roots of $G$ are $t=1,-2,-1/2$. On the other hand, the polynomials $B_2(\textbf{y},1), B_2(\textbf{y},-2)$, $B_2(\textbf{y},-1/2)\in\mathbb{Q}[\textbf{y}]$, in $\textbf{y}$ alone, must also vanish at $y$. Therefore, after computing the possible rational roots of these polynomials with {\tt{Magma}}, we deduce that 
\[y\in\{\pm{3/2},\pm{5/2}\}\;\;\; \text{and}\;\;\; t\in\{1,-1/2,-2\}.\]  
Finally, we must have that $B_i(y,t)=0$ for all $1\leq i\leq 6$ for these $12$ possible pairs $(y,t)$. In particular, we obtain possible finite orbit tuples $(c_1,c_2,P)=(-5/16,-29/16, 5/4)$ and $(c_1,c_2,P)=(-21/16, -29/16,P)$ for several $P$'s, including $P=-1/4$, which we check with {\tt{Magma}} is a point of finite orbit for $(c_1,c_2)=(-21/16, -29/16)$. Moreover, we can rule out the tuple $(c_1,c_2,P)=(-5/16,-29/16, 5/4)$, since in this case the point $Q:=f_2(f_2(Q))$ satisfies $f_1^4(Q)\neq f_1^2(Q)$, and hence $Q$ cannot be preperiodic for $f_1$; see \cite{Poonen}. In particular, $P$ cannot have finite orbit for $S$, a contradiction.   
\end{proof} 
Finally, we prove a classification result for sets of quadratic polynomials $S=\{f_1,f_2\}$ possessing an unspecified finite orbit point when $f_1$ has a rational point of period two and $f_2$ has a rational point of period three.
\begin{lemma}{\label{lem:23}} Let $S=\{f_1, f_2\}$ with $f_1=x^2+c_1$ and $f_2=x^2+c_2$ for some distinct $c_1,c_2\in\mathbb{Q}$. If $S$ has the following properties:     
\begin{enumerate}[topsep=7pt, partopsep=7pt, itemsep=8pt]
\item[\textup{(1)}] $\mu_S(\mathbb{Q})\leq3$,  
\item[\textup{(2)}] $f_1$ has a rational point of period two and $f_2$ has a rational point of period three, 
\item[\textup{(3)}] there exists $P\in\mathbb{Q}$ with finite orbit for $S$,    
\end{enumerate}  
then $c_1=-\frac{21}{16}$ and $c_2=-\frac{29}{16}$.     
\end{lemma}
\begin{proof}
Since $f_1$ has a rational point of period two, we may write $c_1=-(3+y^2)/4$ for some $y\in\mathbb{Q}$. Likewise, since $f_2$ has a rational point of period three, we may parametrize $f_2$ and the $3$-cycle $(c_2,P_1,P_2,P_3)$ as in (\ref{3cycle}); however, we use the variable $t$ to parametrize $f_2$ and the $3$-cycle in this case. Now suppose that $P\in\mathbb{Q}$ has finite orbit for $S$. Then $P$ is preperiodic for $f_2$. In particular, it follows from Theorems 1, 2, and 3 of \cite{Poonen} that the $f_2$-orbit of $P$ contains the $3$-cycle $P_1$,$P_2$ and $P_3$; here we use that $\mu_S(\mathbb{Q})\leq3$. Hence, the $S$-orbit of $P$ contains $P_1$. Therefore, we may assume that $P=P_1$ has finite orbit for $S$.  

Now, since $P_1$ has finite orbit for $S$, it follows that $P_1$ is preperiodic for $f_1$. Hence, Theorems 1, 2, and 3 of \cite{Poonen} imply that $f_1^4(P_1)=f_1^2(P_1)$; here we use that $f_1$ has a $2$-cycle and that $\mu_S(\mathbb{Q})\leq3$. Likewise, $Q:=f_2(P_1)$ is preperiodic for $f_1$ and $f_1^4(Q)=f_1^2(Q)$. Now let $N(\textbf{y},\textbf{t})$ be the polynomial in two variables given by the numerator of the rational expression $f_1^4(P_1)-f_1^2(P_1)$, and let $A(\textbf{y},\textbf{t})$ be the numerator of the expression $f_1^4(Q)-f_1^2(Q)$. In particular, we have that 
\[N(y,t)=0\;\;\; \text{and}\;\;\; A(y,t)=0,\] 
for all pairs $(y,t)$ associated to finite orbit tuples $(c_1,c_2,P)$ as in Lemma \ref{lem:23}. On the other hand, we compute a Gr\"{o}bner basis $\{B_1(\textbf{y},\textbf{t}),B_2(\textbf{y},\textbf{t}),\dots B_6(\textbf{y},\textbf{t})\}$ for the $\mathbb{Q}[\textbf{y},\textbf{t}]$ ideal $I$ generated by $N(\textbf{y},\textbf{t})$ and $A(\textbf{y},\textbf{t})$. In particular, we see that $I$ contains a polynomial $G(\textbf{t})\in\mathbb{Q}[\textbf{t}]$ (in $\textbf{t}$ alone) of degree $68$. In particular, if $(y,t)$ is associated to a finite orbit tuple $(c_1,c_2,P)$ as in Lemma \ref{lem:23}, then $G(t)=0$. However, we compute with {\tt{Magma}} that the only rational roots of $G$ are $t=1,-2,-1/2$. On the other hand, the polynomials $B_2(\textbf{y},1), B_2(\textbf{y},-2)$, $B_2(\textbf{y},-1/2)\in\mathbb{Q}[\textbf{y}]$, in $\textbf{y}$ alone, must also vanish at $y$. Therefore, after computing the possible rational roots of these polynomials with {\tt{Magma}}, we deduce that 
\[y\in\{\pm{1/2},\pm{3/2},\pm{5/2}\}\;\;\; \text{and}\;\;\; t\in\{1,-1/2,-2\}.\] 
Finally, we must have that $B_i(y,t)=0$ for all $1\leq i\leq 6$ for these $18$ possible pairs $(y,t)$. In particular, we obtain possible finite orbit tuples $(c_1,c_2,P)=(-37/16, -29/16, -7/4)$, $(c_1,c_2,P)=(-13/16, -29/16)$ and $(c_1,c_2,P)=(-21/16, -29/16,P)$ for several $P$'s, including $P=-1/4$, which we check with {\tt{Magma}} is a point of finite orbit for $(c_1,c_2)=(-21/16, -29/16)$. Moreover, we can rule out the tuples $(c_1,c_2,P)=(-5/16,-29/16, 5/4)$ and $(c_1,c_2,P)=(-13/16, -29/16)$. For in both cases, the points $Q_1=f_2^2(P)$ satisfy $f_1^4(Q)\neq f_1^2(Q)$, and hence $Q$ cannot be preperiodic for $f_1$; see \cite{Poonen}. In particular, $P$ cannot have finite orbit for $S$, a contradiction. 
\end{proof} 
Having classified certain pairs of quadratic polynomials with finite orbit points, we use these results to prove our main theorem. 
\section{Proof of Main Theorem} 
\begin{proof}[(Proof of Theorem \ref{thm:classification})] Suppose that $S=\{x^2+c_1,x^2+c_2,x^2+c_3\}$ for some distinct rational numbers $c_i$, that $\mu_S(\mathbb{Q})\leq3$, and that $P\in\mathbb{Q}$ is a finite orbit point for $S$. In particular, the orbit of $P$ is preperiodic for each individual map in $S$. Hence, the orbit of $P$ enters a fixed point, a $2$-cycle, or a $3$-cycle for each map is $S$ (since $\mu_S(\mathbb{Q})\leq3$ by assumption). In what follows, we write $\phi_i(x)=x^2+c_i$ for all $1\leq i\leq3$. From here we proceed in cases. For the relevant codes, see the file named \emph{Finite Orbit\,-\,Subcases} at the website mentioned in the introduction.  
\\[5pt]
\textbf{Case(1):} Suppose that the orbit of $P$ enters a fixed point for each individual map in $S$. In particular, each map in $S$ has a rational fixed point, and by replacing $P$ with $\phi_1^n(P)$ for some $n$, we may assume that a fixed point of $\phi_1$, renamed $P$, is a finite orbit point for both sets $S'=\{\phi_1,\phi_2\}$ and $S''=\{\phi_1,\phi_3\}$ simultaneously. Then the classification Lemma \ref{lem:11} leads us to several subcases:  
\\[7pt]
\emph{Subcase(1.1)}: The pair of tuples $(c_1, c_2, P)$ and $(c_1,c_3,P)$ take the form: 
\begin{align*}
(c_1, c_2,P)&=\bigg(\frac{1-y^2}{4},\;\frac{1-(y+2)^2}{4},\;\frac{1+y}{2}\bigg),\\[4pt]
(c_1, c_3,P)&=\bigg(\frac{1-u^2}{4},\;\frac{1-(u+2)^2}{4},\;\frac{1+u}{2}\bigg),
\end{align*} 
for some $y,u\in\mathbb{Q}$. Equating expressions for $P$, we deduce that $y=u$. In particular, it follows that $c_2=c_3$, a contradiction.     
\\[7pt]
\emph{Subcase(1.2)}: The pair of tuples $(c_1, c_2, P)$ and $(c_1,c_3,P)$ take the form:
\begin{align*}
(c_1, c_2,P)&=\bigg(\frac{1-y^2}{4},\frac{1-(y+2)^2}{4},\frac{1+y}{2}\bigg),\\[3pt] 
(c_1,c_3,P)&=\bigg(\frac{t^4 - 18t^2 + 1}{4(t^2 -1)^2},\frac{-3t^4 - 10t^2 -3}{4(t^2 -1)^2},\frac{t^2 + 4t - 1}{2(t^2 - 1)}\bigg),
\end{align*}  
for some $y,t\in\mathbb{Q}$. Equating expressions for $P$, we deduce that $y=4t/(t^2-1)$. Therefore, \vspace{.05cm} 
\begin{equation*}
(c_1,c_2,c_3,P)=\Scale[1.25]{\Big(\frac{t^4 - 18t^2 + 1}{4(t^2 -1)^2}, \frac{-3t^4 - 16t^3 - 10t^2 + 16t- 3}{4(t^2 -1)^2},\frac{-3t^4 - 10t^2 -3}{4(t^2 -1)^2},\frac{t^2 + 4t - 1}{2(t^2 - 1)}\Big)} \vspace{.05cm}   
\end{equation*}
is a finite orbit tuple for some $t\in\mathbb{Q}$. In particular, the point $Q:=(\phi_2\circ\phi_2\circ\phi_3\circ\phi_1\circ\phi_3\circ\phi_3)(P)$ is preperiodic for every map in $S$; hence, we may assume that $Q$ is preperiodic for $\phi_3$. However, Theorems 1, 2, and 3 of \cite{Poonen} together imply that $\phi_3^4(\alpha)=\phi^2(\alpha)$ for all $\alpha\in\PrePer(\phi_3,\mathbb{Q})$; equivalently, after two iterates of $\phi_3$, any rational preperiodic point of $\phi_3$ enters a $2$-cycle or a fixed point. We use crucially here our assumption that $\mu_S(\mathbb{Q})\leq3$. In particular, $\phi_3^4(Q)-\phi^2(Q)=0$. However, this relation forces $t$ to be a rational root of a very large degree polynomial, and we compute with {\tt{Magma}} that $t=0$ is the  only rational root of this polynomial. Therefore, $t=0$ and $c_2=-3/4=c_3$, a contradiction.   
\\[7pt]
\emph{Subcase(1.3)}: The pair of tuples $(c_1, c_2, P)$ and $(c_1,c_3)$ take the form:
\begin{align*} 
(c_1, c_2,P)=\bigg(\frac{1-y^2}{4},\frac{1-(y+2)^2}{4},\frac{1+y}{2}\bigg), \hspace{.5cm} (c_1, c_3)\in\bigg\{\Big(-\frac{21}{16},-\frac{5}{16}\Big), \Big(\frac{3}{16},\frac{-5}{16}\Big)\bigg\}, 
\end{align*}
for some $y\in\mathbb{Q}$. Equating expressions for $c_1$, we see that $y=\pm{5/2}$ or $y=\pm{1/2}$. If $y=5/2$, then $(c_1,c_2,c_3)=\big(-\frac{21}{16},-\frac{77}{16},-\frac{5}{16}\big)$. Let $P_1=7/4$ and $P_2=-3/4$  be the fixed points of $\phi_1$, one of which has finite orbit for $S$ by assumption. However, the points $Q_1=\phi_3(P_1)$ and $Q_2=\phi_2(P_2)$ satisfy $\phi_1^4(Q_i)\neq\phi_1^2(Q_i)$ for both $i=1,2$; hence, $Q_1$ and $Q_2$ cannot be preperiodic for $\phi_1$; see \cite{Poonen}. In particular, $Q_1$ and $Q_2$ are not finite orbit points for $S$. However, this contradicts the fact that either $P_1$ or $P_2$ has finite orbit for $S$. The cases $y=-5/2,\pm{1/2}$ follow similarly; see our {\tt{Magma}} code for Subcase 1.3.   
\\[7pt] 
\emph{Subcase(1.4)}: The pair of tuples $(c_1, c_2, P)$ and $(c_1,c_3,P)$ take the form: \vspace{.15cm}  
\begin{align*}
(c_1,c_2,P)&=\bigg(\frac{t^4 - 18t^2 + 1}{4(t^2 -1)^2},\, \frac{-3t^4 - 10t^2 -3}{4(t^2 -1)^2},\, \frac{t^2 + 4t - 1}{2(t^2 - 1)}\bigg),\\[4pt]
(c_1, c_3,P)&=\bigg(\frac{1-y^2}{4},\frac{1-(y+2)^2}{4},\frac{1+y}{2}\bigg),
\end{align*} 
for some $t,y\in\mathbb{Q}$. Interchanging $c_2$ and $c_3$, we may proceed as in Subcase 1.2.       
\\[7pt] 
\emph{Subcase(1.5)}: The pair of tuples $(c_1, c_2, P)$ and $(c_1,c_3,P)$ take the form: \vspace{.15cm} 
\begin{align*}
(c_1,c_2,P)&=\bigg(\frac{t^4 - 18t^2 + 1}{4(t^2 -1)^2},\, \frac{-3t^4 - 10t^2 -3}{4(t^2 -1)^2},\, \frac{t^2 + 4t - 1}{2(t^2 - 1)}\bigg),\\[5pt] 
(c_1,c_3,P)&=\bigg(\frac{u^4 - 18u^2 + 1}{4(u^2 -1)^2},\, \frac{-3u^4 - 10u^2 -3}{4(u^2 -1)^2},\, \frac{u^2 + 4u - 1}{2(u^2 - 1)}\bigg),
\end{align*}
for some $t,u\in\mathbb{Q}$. Equating expressions for $P$, we see that $u=t$ or $u=-1/t$. However, in either case, we obtain that $c_2=c_3$, a contradiction.              
\\[7pt] 
\emph{Subcase(1.6)}: The pair of tuples $(c_1, c_2, P)$ and $(c_1,c_3)$ take the form: \vspace{.15cm}
\begin{align*}
(c_1,c_2,P)&=\bigg(\frac{t^4 - 18t^2 + 1}{4(t^2 -1)^2},\, \frac{-3t^4 - 10t^2 -3}{4(t^2 -1)^2},\, \frac{t^2 + 4t - 1}{2(t^2 - 1)}\bigg),\\[5pt] 
(c_1, c_3)&\in\bigg\{\Big(-\frac{21}{16},-\frac{5}{16}\Big), \Big(\frac{3}{16},\frac{-5}{16}\Big)\bigg\}, 
\end{align*} 
for some $t\in\mathbb{Q}$. However, equating possible expressions for $c_1$, we obtain polynomials over $\mathbb{Q}$ with no rational roots, a contradiction.                 
\\[7pt] 
\emph{Subcase(1.7)}: The pair of tuples $(c_1, c_2)$ and $(c_1,c_3,P)$ take the form: \vspace{.15cm}
\[(c_1, c_2)\in\bigg\{\Big(-\frac{21}{16},-\frac{5}{16}\Big), \Big(\frac{3}{16},\frac{-5}{16}\Big)\bigg\}, \hspace{.5cm} (c_1, c_3,P)=\bigg(\frac{1-y^2}{4},\frac{1-(y+2)^2}{4},\frac{1+y}{2}\bigg),\] 
for some $y\in\mathbb{Q}$. However, by interchanging $c_2$ and $c_3$, we may proceed as in Subcase 1.3.                   
\\[7pt] 
\emph{Subcase(1.8)}: The pair of tuples $(c_1, c_2)$ and $(c_1,c_3,P)$ take the form: \vspace{.15cm} 
\begin{align*} 
(c_1, c_2)&\in\bigg\{\Big(-\frac{21}{16},-\frac{5}{16}\Big), \Big(\frac{3}{16},\frac{-5}{16}\Big)\bigg\}, \\[5pt] 
(c_1, c_3,P)&=\bigg(\frac{t^4 - 18t^2 + 1}{4(t^2 -1)^2},\, \frac{-3t^4 - 10t^2 -3}{4(t^2 -1)^2},\, \frac{t^2 + 4t - 1}{2(t^2 - 1)}\bigg),
\end{align*}  
for some $t\in\mathbb{Q}$. However, by interchanging $c_2$ and $c_3$, we may proceed as in Subcase 1.6.                   
\\[7pt] 
\emph{Subcase(1.9)}: The pair of tuples $(c_1,c_2)$ and $(c_1,c_3)$ take the form: 
\[(c_1, c_2)\in\bigg\{\Big(-\frac{21}{16},-\frac{5}{16}\Big), \Big(\frac{3}{16},\frac{-5}{16}\Big)\bigg\} \hspace{.5cm} (c_1, c_3)\in\bigg\{\Big(-\frac{21}{16},-\frac{5}{16}\Big), \Big(\frac{3}{16},\frac{-5}{16}\Big)\bigg\}.\vspace{.1cm}\]     
However in every cases, we see that $c_2=-5/16=c_3$, a contradiction.            
\\[7pt] 
\textbf{Case(2):} Suppose that the orbit of $P$ enters a fixed point for two maps in $S$, say $\phi_1$ and $\phi_2$, and a $2$-cycle for $\phi_3$. In particular, both $\phi_1$ and $\phi_2$ have rational fixed points and $\phi_3$ has a rational point of period $2$. Moreover, by replacing $P$ with $\phi_1^n(P)$ for some $n$, we may assume that a fixed point of $\phi_1$, renamed $P$, has finite orbit for both sets $S'=\{\phi_1, \phi_2\}$ and $S''=\{\phi_1,\phi_3\}$ simultaneously. Then Lemma \ref{lem:11} and Lemma \ref{lem:12} together imply several subcases: 
\\[7pt] 
\emph{Subcase(2.1)}: The pair of tuples $(c_1, c_2, P)$ and $(c_1,c_3,P)$ take the form: \vspace{.1cm}  
\begin{align*}
(c_1, c_2,P)&=\bigg(\frac{1-y^2}{4},\frac{1-(y+2)^2}{4},\frac{1+y}{2}\bigg), \\[4pt] 
(c_1,c_3,P)&=\bigg(\frac{1-u^2}{4},\,\frac{-3-u^2}{4},\, \frac{1+u}{2}\bigg),
\end{align*}  
for some $y,u\in\mathbb{Q}$. Equating expressions for $P$, we see that $y=u$. Hence, 
\[ (c_1,c_2,c_3,P)=\bigg(\frac{1-y^2}{4},\,\frac{1-(y+2)^2}{4},\,\frac{-3-y^2}{4},\,\frac{1+y}{2}\bigg)\]
for some $y\in\mathbb{Q}$. Now, if $P$ has finite orbit for $S$, then so does $Q:=(\phi_1\circ\phi_1\circ\phi_1\circ\phi_2\circ\phi_3)(P)$. In particular, $Q$ must be preperiodic for $\phi_2$. However, since $\mu_S(\mathbb{Q})\leq3$ and $\phi_2$ has a rational fixed point, Theorems 1, 2, and 3 of \cite{Poonen} together imply that $\phi_2^4(Q)=\phi_2^2(Q)$ (equivalently, any rational preperiodic point of $\phi_2$ enters a 2-cycle or fixed point after two iterates). However, the relation $\phi_2^4(Q)-\phi_2^2(Q)=0$ forces $y$ to be a root of a polynomial of large degree, and we compute with {\tt{Magma}} that the only rational roots of this polynomial are $y=0,-1,-1/2$. On the other hand, if $y=0$, then we see immediately that $c_2=c_3$, a contradiction. Similarly, if $y=-1$, then $c_1=c_2$, a contradiction. Finally, if we assume that $y=-1/2$, then $(c_1,c_2,c_3,P)=(3/16, -5/16, -13/16, 1/4)$, and it is straightforward to verify that this tuple has finite orbit.        
\\[7pt] 
\emph{Subcase(2.2)}: The pair of tuples $(c_1, c_2, P)$ and $(c_1,c_3,P)$ take the form: \vspace{.1cm} 
\begin{align*} 
(c_1, c_2,P)&=\bigg(\frac{1-y^2}{4},\frac{1-(y+2)^2}{4},\frac{1+y}{2}\bigg), \\[4pt] 
(c_1,c_3,P)&=\bigg(\frac{-15t^4 - 2t^2 + 1}{4(t^2 -1)^2},\, \frac{-3t^4 - 10t^2 -3}{4(t^2 -1)^2},\, \frac{-3t^2- 1}{2(t^2 - 1)}\bigg), 
\end{align*}  
for some $y,t\in\mathbb{Q}$. Equating expressions for $P$, we see that $y=-4t^2/(t^2-1)$. However, this relation forces $c_2=c_3$, a contradiction. 
\\[7pt] 
\emph{Subcase(2.3)}: The pair of tuples $(c_1, c_2, P)$ and $(c_1,c_3)$ take the form:
\[(c_1, c_2,P)=\bigg(\frac{1-y^2}{4},\frac{1-(y+2)^2}{4},\frac{1+y}{2}\bigg), \hspace{.5cm} (c_1, c_3)=\Big(-\frac{5}{16},-\frac{13}{16}\Big),\] 
for some $y\in\mathbb{Q}$. Equating expressions for $c_1$, we see that $y=\pm{3/2}$. Suppose first that $y=3/2$, so that $(c_1,c_2,c_3,P)=(-\frac{21}{16},-\frac{77}{16},-\frac{13}{16},\frac{5}{4})$. However, since $P$ has finite orbit for $S$, so does $Q:=(\phi_1\circ\phi_2\circ\phi_3)(P)$. Hence, $Q$ must be preperiodic for $\phi_2$. In particular, since $\phi_2$ has a rational fixed point and $\mu_S(\mathbb{Q})\leq3$, it follows from Theorems 1, 2 and 3 of \cite{Poonen} that $\phi_2^4(Q)=\phi_2^2(Q)$. However, this is not the case. On the other hand, if $y=-3/2$, then $(c_1,c_2,c_3)=(-5/16, 3/16, -13/16)$. However, this tuple is known to have a finite orbit rational point; see Subcase 2.1 above.    
\\[7pt] 
\emph{Subcase(2.4)}: The pair of tuples $(c_1, c_2, P)$ and $(c_1,c_3)$ take the form:
\[(c_1, c_2,P)=\bigg(\frac{1-y^2}{4},\frac{1-(y+2)^2}{4},\frac{1+y}{2}\bigg), \hspace{.5cm} (c_1, c_3)=\Big(-\frac{21}{16},-\frac{13}{16}\Big),\] 
for some $y\in\mathbb{Q}$. Equating expressions for $c_1$, we see that $y=\pm{5/2}$. Suppose first that $y=5/2$, so that $(c_1,c_2,c_3,P)=(-\frac{5}{16},-\frac{45}{16},-\frac{13}{16},\frac{7}{4})$. However, since $P$ has finite orbit for $S$, so does $Q:=\phi_3(P)$. Hence, $Q$ must be preperiodic for $\phi_2$. In particular, since $\phi_2$ has a rational fixed point and $\mu_S(\mathbb{Q})\leq3$, it follows from Theorems 1, 2 and 3 of \cite{Poonen} that $\phi_2^4(Q)=\phi_2^2(Q)$. However, this is not the case. Likewise, if $y=-5/2$, then $(c_1,c_2,c_3,P)=(-\frac{21}{16},\frac{3}{16},-\frac{13}{16},-\frac{3}{4})$ and $Q:=(\phi_1\circ\phi_3)(P)$ must be preperiodic for $\phi_2$. However, this is not the case, since $\phi_2^4(Q)\neq\phi_2^2(Q)$.  
\\[7pt] 
\emph{Subcase(2.5)}: The pair of tuples $(c_1, c_2, P)$ and $(c_1,c_3,P)$ take the form: \vspace{.1cm}  
\begin{align*}
(c_1, c_2,P)&=\bigg(\frac{t^4 - 18t^2 + 1}{4(t^2 -1)^2},\, \frac{-3t^4 - 10t^2 -3}{4(t^2 -1)^2},\, \frac{t^2 + 4t - 1}{2(t^2 - 1)}\bigg), \\[4pt] 
(c_1,c_3,P)&=\bigg(\frac{1-u^2}{4},\,\frac{-3-u^2}{4},\, \frac{1+u}{2}\bigg),
\end{align*}  
for some $t,u\in\mathbb{Q}$. Equating expressions for $P$, we see that $u=4t/(t^2-1)$. However, this relation forces $c_2=c_3$, a contradiction.  
\\[7pt]
\emph{Subcase(2.6)}: The pair of tuples $(c_1, c_2, P)$ and $(c_1,c_3,P)$ take the form: \vspace{.1cm} 
\begin{align*} 
(c_1, c_2,P)&=\bigg(\frac{t^4 - 18t^2 + 1}{4(t^2 -1)^2},\frac{-3t^4 - 10t^2 -3}{4(t^2 -1)^2},\frac{t^2 + 4t - 1}{2(t^2 - 1)}\bigg), \\[4pt] 
(c_1,c_3,P)&=\bigg(\frac{-15u^4 - 2u^2 + 1}{4(u^2 -1)^2},\, \frac{-3u^4 - 10u^2 -3}{4(u^2 -1)^2},\, \frac{-3u^2- 1}{2(u^2 - 1)}\bigg), 
\end{align*}  
for some $t,u\in\mathbb{Q}$. Equating expressions for $P$, we see that the pair $(t,u)$ is a rational point on the affine curve 
\[C: t^2u^2 + tu^2 - t - u^2=0.\]
However, the map $r: (t,u)\rightarrow((-tu^2 + 1)/u^2,\;(tu^2 + u^2 - 1)/u^3)$ determines a degree one-map to the elliptic curve $E$ with Weierstrass equation $E: y^2 = x^3 - 2x^2 + 1$. Moreover, we compute with {\tt{Magma}} that $E(\mathbb{Q})$ has rank zero, that $E(\mathbb{Q})\cong\mathbb{Z}/6\mathbb{Z}$, and that 
\[E(\mathbb{Q})=\{\infty, (1,0), (0,\pm{1}), (2,\pm{1})\}.\] 
On the other hand, it is straightforward to check that there are no points $p\in C(\mathbb{Q})$ such that $r(p)=(1,0)$, and we deduce that 
\[C(\mathbb{Q})=\{(0,0), (\pm{1},\pm{1})\}.\]
On the other hand, if $(t,u)\in\{(\pm{1},\pm{1})\}$, then at least one of the values $c_1, c_2$ or $c_3$ is infinite, a contradiction. Finally, if $(t,u)=(0,0)$, then $c_2=-3/4=c_3$, a contradiction.   
\\[7pt] 
\emph{Subcase(2.7)}: The pair of tuples $(c_1, c_2, P)$ and $(c_1,c_3)$ take the form:
\[(c_1, c_2,P)=\bigg(\frac{t^4 - 18t^2 + 1}{4(t^2 -1)^2},\frac{-3t^4 - 10t^2 -3}{4(t^2 -1)^2},\frac{t^2 + 4t - 1}{2(t^2 - 1)}\bigg), \hspace{.5cm} (c_1, c_3)=\Big(-\frac{5}{16},-\frac{13}{16}\Big),\]  
for some $t\in\mathbb{Q}$. Equating expressions for $c_1$, we see that $t\in\{\pm{3},\pm{1/3}\}$. On the other hand, in all four cases, $(c_1, c_2,c_3)=(-\frac{5}{16}, -\frac{21}{16}, -\frac{13}{16})$, and it is straight forward to check that this tuple has finite orbit rational points $P\in\big\{\pm{1/4},\pm{3/4},\pm{5/4}\big\}$. Moreover, these are the only finite orbit rational points for this tuple.   
\\[7pt]
\emph{Subcase(2.8)}: The pair of tuples $(c_1, c_2, P)$ and $(c_1,c_3)$ take the form:
\[(c_1, c_2,P)=\bigg(\frac{t^4 - 18t^2 + 1}{4(t^2 -1)^2},\frac{-3t^4 - 10t^2 -3}{4(t^2 -1)^2},\frac{t^2 + 4t - 1}{2(t^2 - 1)}\bigg), \hspace{.5cm} (c_1, c_3)=\Big(-\frac{21}{16},-\frac{13}{16}\Big),\]  
for some $t\in\mathbb{Q}$. However, there are no rational solutions $t$ to the equation $\frac{t^4 - 18t^2 + 1}{4(t^2 -1)^2}=-\frac{21}{16}$, and we obtain a  contradiction.   
\\[7pt]
\emph{Subcase(2.9)}: The pair of tuples $(c_1, c_2)$ and $(c_1,c_3,P)$ take the form: \vspace{.15cm}
\[(c_1, c_2)\in\bigg\{\Big(-\frac{21}{16},-\frac{5}{16}\Big), \Big(\frac{3}{16},\frac{-5}{16}\Big)\bigg\}, \hspace{.5cm} (c_1,c_3,P)=\bigg(\frac{1-u^2}{4},\,\frac{-3-u^2}{4},\, \frac{1+u}{2}\bigg),\] 
for some $u\in\mathbb{Q}$. Equating expressions for $c_1$, we see that $u=\pm{5/2},\pm{1/2}$. If $u=5/2$, then $(c_1,c_2,c_3,P)=(-\frac{21}{16},-\frac{5}{16},-\frac{37}{16}, \frac{7}{4})$. However, if $P$ has finite orbit for $S$, then so does $Q:=\phi_2(P)$. In particular, $Q$ must be preperiodic for $\phi_3$. Hence, Theorems 1, 2, and 3 of \cite{Poonen} together imply that $\phi_3^4(Q)=\phi_3^2(Q)$. However, this is not the case. Likewise, if $u=-5/2$, then $(c_1,c_2,c_3,P)=(-\frac{21}{16},-\frac{5}{16},-\frac{37}{16}, -\frac{3}{4})$, and the point $Q:=\phi_2(P)$ must satisfy $\phi_3^4(Q)=\phi_3^2(Q)$; again, this is not the case. On the other hand, if $u=\pm{1/2}$, then $(c_1,c_2,c_3)=(\frac{3}{16},-\frac{5}{16},-\frac{13}{16})$ and $P=1/4$ or $P=3/4$. Moreover, it is straightforward to check that $P$ has finite orbit for $S$ in both cases.  
\\[7pt] 
\emph{Subcase(2.10)}: The pair of tuples $(c_1, c_2)$ and $(c_1,c_3,P)$ take the form: \vspace{.15cm}
\begin{align*} 
(c_1, c_2)&\in\bigg\{\Big(-\frac{21}{16},-\frac{5}{16}\Big), \Big(\frac{3}{16},\frac{-5}{16}\Big)\bigg\},\\[5pt]
(c_1,c_3, P)&=\bigg(\frac{-15t^4 - 2t^2 + 1}{4(t^2 -1)^2},\, \frac{-3t^4 - 10t^2 -3}{4(t^2 -1)^2},\, \frac{-3t^2- 1}{2(t^2 - 1)}\bigg),
\end{align*} 
for some $t\in\mathbb{Q}$. However, there are no rational solutions $t$ to the equations $\frac{-15t^4 - 2t^2 + 1}{4(t^2 -1)^2}=-\frac{21}{16}$. Hence, the case of $(c_1,c_2)=(-21/16,-5/16)$ is impossible. On the other hand, the equation $\frac{-15t^4 - 2t^2 + 1}{4(t^2 -1)^2}=\frac{3}{16}$ forces $t=\pm{1/3}$. Therefore, we may assume that the tuple $S$ and the point $P$ are given by $(c_1,c_2,c_3,P)=(\frac{3}{16}, -\frac{5}{16}, -\frac{21}{16}, \frac{3}{4})$. In this case, if $P$ has finite orbit for $S$, then so does $Q:=\phi_3(\phi_2(P))$. In particular, $Q$ must be preperiodic for $\phi_1$. Hence, Theorems 1, 2, and 3 of \cite{Poonen} imply that $\phi_1^4(Q)=\phi_1^2(Q)$. However, this is not the case, a contradiction.  
\\[7pt] 
\emph{Subcase(2.11)}: The pair of tuples $(c_1, c_2)$ and $(c_1,c_3)$ take the form: \vspace{.15cm}
\begin{align*}
(c_1, c_2)&\in\bigg\{\Big(-\frac{21}{16},-\frac{5}{16}\Big), \Big(\frac{3}{16},\frac{-5}{16}\Big)\bigg\},\\[5pt]
(c_1,c_3)&\in\bigg\{\Big(-\frac{5}{16},-\frac{13}{16}\Big),\Big(-\frac{21}{16},-\frac{13}{16}\Big) \bigg\}.  
\end{align*} 
However, after equating possible expressions for $c_1$, we see that $(c_1,c_2,c_3)=(-\frac{21}{16},-\frac{5}{16},-\frac{13}{16})$ is the only case to be considered. Moreover, this tuple is known (up to reordering) to have rational points with finite orbit.       
\\[7pt]   
\textbf{Case(3):} Suppose that the orbit of $P$ enters a fixed point for two of the maps in $S$, say $\phi_1$ and $\phi_2$, and a $3$-cycle for $\phi_3$. Then in particular, both $\phi_1$ and $\phi_2$ have rational fixed points and $\phi_3$ has a rational point of period $3$. Moreover, $P$ has finite orbit for both sets $S'=\{\phi_1, \phi_3\}$ and $S''=\{\phi_2,\phi_3\}$ simultaneously. Hence, Lemma \ref{lem:13} applied to $S'$ implies that $c_1=-21/16$ and $c_3=-29/16$. Likewise, Lemma \ref{lem:13} applied to $S''$ implies that $c_2=-21/16$ and $c_3=-29/16$. Therefore, we deduce that $c_1=c_2$, a contradiction.     
\\[7pt]
\textbf{Case(4):} Suppose that the orbit of $P$ enters a fixed point for only one map in $S$, say $\phi_1$, and enters a $2$-cycle for the others, $\phi_2$ and $\phi_3$. In particular, $\phi_1$ has a rational fixed point and $\phi_2$ and $\phi_3$ have rational points of period $2$. Moreover, by replacing $P$ with $\phi_1^n(P)$ for some $n$, we may assume that a fixed point of $\phi_1$, renamed $P$, has finite orbit for both sets $S'=\{\phi_1,\phi_2\}$ and $S''=\{\phi_1,\phi_3\}$ simultaneously. Then the classification Lemma \ref{lem:12} implies several subcases: 
\\[7pt]
\emph{Subcase(4.1)}: The pair of tuples $(c_1, c_2, P)$ and $(c_1,c_3,P)$ take the form: \vspace{.1cm}  
\begin{align*}
(c_1, c_2,P)&=\bigg(\frac{1-y^2}{4},\,\frac{-3-y^2}{4},\, \frac{1+y}{2}\bigg), \\[4pt] 
(c_1,c_3,P)&=\bigg(\frac{1-u^2}{4},\,\frac{-3-u^2}{4},\, \frac{1+u}{2}\bigg),
\end{align*}  
for some $y,u\in\mathbb{Q}$. Equating expressions for $P$, we see that $y=u$. However, this immediately implies that $c_2=c_3$, a contradiction.
\\[7pt]
\emph{Subcase(4.2)}: The pair of tuples $(c_1, c_2, P)$ and $(c_1,c_3,P)$ take the form: \vspace{.1cm}  
\begin{align*}
(c_1, c_2,P)&=\bigg(\frac{1-y^2}{4},\,\frac{-3-y^2}{4},\, \frac{1+y}{2}\bigg), \\[4pt] 
(c_1,c_3,P)&=\bigg(\frac{-15t^4 - 2t^2 + 1}{4(t^2 -1)^2},\, \frac{-3t^4 - 10t^2 -3}{4(t^2 -1)^2},\, \frac{-3t^2- 1}{2(t^2 - 1)}\bigg),
\end{align*}  
for some $y,t\in\mathbb{Q}$. Equating expressions for $P$, we see that $y=-4t^2/(t^2-1)$. In particular, \vspace{.1cm} 
\[(c_1,c_2,c_3,P)=\bigg(\frac{-15t^4 - 2t^2 + 1}{4(t^2 -1)^2},\, \frac{-19t^4 + 6t^2 - 3}{4(t^2 -1)^2} ,\, \frac{-3t^4 - 10t^2 -3}{4(t^2 -1)^2},\, \frac{-3t^2- 1}{2(t^2 - 1)}\bigg)\]
for some $t\in\mathbb{Q}$. On the other hand, since $P$ has finite orbit for $S$, so does $Q:=(\phi_1^2\circ\phi_2\circ\phi_3)(P)$. In particular, $Q$ must be preperiodic for $\phi_3$, and it follows from Theorems 1, 2, and 3 of \cite{Poonen} that $\phi_3^4(Q)-\phi_3^2(Q)=0$. However, the only rational solution to this equation (in one variable) is $t=0$. Therefore, we may assume that $t=0$. Hence, $(c_1,c_2,c_3,P)=(\frac{1}{4}, -\frac{3}{4}, -\frac{3}{4}, \frac{1}{2})$, and $c_2=c_3$, a contradiction.     
\\[7pt]
\emph{Subcase(4.3)}: The pair of tuples $(c_1, c_2,P)$ and $(c_1,c_3)$ take the form: \vspace{.15cm}
\[(c_1, c_2,P)=\bigg(\frac{1-y^2}{4},\,\frac{-3-y^2}{4},\, \frac{1+y}{2}\bigg), \hspace{.5cm} (c_1, c_3)=\Big(-\frac{5}{16},-\frac{13}{16}\Big).\]
Equating expressions for $c_1$, we see that $y=\pm{3/2}$; hence, $(c_1,c_2,c_3,P)=(-\frac{5}{16},-\frac{21}{16},-\frac{13}{16},\frac{5}{4})$ or $(c_1,c_2,c_3,P)=(-\frac{5}{16},-\frac{21}{16},-\frac{13}{16},-\frac{1}{4})$. Moreover, it is straightforward to check that these tuples do indeed have finite orbit.    
\\[7pt] 
\emph{Subcase(4.4)}: The pair of tuples $(c_1, c_2,P)$ and $(c_1,c_3)$ take the form: \vspace{.15cm}
\[(c_1, c_2,P)=\bigg(\frac{1-y^2}{4},\,\frac{-3-y^2}{4},\, \frac{1+y}{2}\bigg), \hspace{.5cm} (c_1, c_3)=\Big(-\frac{21}{16},-\frac{13}{16}\Big).\]
Equating expressions for $c_1$, we see that $y=\pm{5/2}$; hence, $(c_1,c_2,c_3,P)=(-\frac{21}{16},-\frac{37}{16},-\frac{13}{16},\frac{7}{4})$ or $(c_1,c_2,c_3,P)=(-\frac{21}{16},-\frac{37}{16},-\frac{13}{16},-\frac{3}{4})$. In either case, since $P$ is assumed to have finite orbit for $S$, the point $Q:=(\phi_2\circ\phi_3\circ\phi_2)(P)$ must have finite orbit for $S$. In particular, $Q$ must be preperiodic for $\phi_3$. Therefore, Theorems 1, 2 and 3 of \cite{Poonen} together imply that $\phi_3^4(Q)-\phi_3^2(Q)=0$. However, this is not true in either case. 
\\[7pt]
\emph{Subcase(4.5)}: The pair of tuples $(c_1, c_2, P)$ and $(c_1,c_3,P)$ take the form: \vspace{.1cm}  
\begin{align*}
(c_1, c_2,P)&=\bigg(\frac{-15t^4 - 2t^2 + 1}{4(t^2 -1)^2},\, \frac{-3t^4 - 10t^2 -3}{4(t^2 -1)^2},\, \frac{-3t^2- 1}{2(t^2 - 1)}\bigg), \\[4pt] 
(c_1,c_3,P)&=\bigg(\frac{1-y^2}{4},\,\frac{-3-y^2}{4},\, \frac{1+y}{2}\bigg),
\end{align*}  
for some $t,y\in\mathbb{Q}$. However, by interchanging $c_2$ and $c_3$, we may proceed as in Subcase 4.2 above.   
\\[7pt]
\emph{Subcase(4.6)}: The pair of tuples $(c_1, c_2, P)$ and $(c_1,c_3,P)$ take the form: \vspace{.1cm}  
\begin{align*}
(c_1, c_2,P)&=\bigg(\frac{-15t^4 - 2t^2 + 1}{4(t^2 -1)^2},\, \frac{-3t^4 - 10t^2 -3}{4(t^2 -1)^2},\, \frac{-3t^2- 1}{2(t^2 - 1)}\bigg), \\[4pt] 
(c_1,c_3,P)&=\bigg(\frac{-15u^4 - 2u^2 + 1}{4(u^2 -1)^2},\, \frac{-3u^4 - 10u^2 -3}{4(u^2 -1)^2},\, \frac{-3u^2- 1}{2(u^2 - 1)}\bigg),
\end{align*}
for some $t,u\in\mathbb{Q}$. Equating expressions for $P$, we see that $t=\pm{u}$. However, this immediately forces $c_2=c_3$, a contradiction.  
\\[7pt] 
\emph{Subcase(4.7)}: The pair of tuples $(c_1, c_2,P)$ and $(c_1,c_3)$ take the form: \vspace{.15cm}
\[(c_1, c_2,P)=\bigg(\frac{-15t^4 - 2t^2 + 1}{4(t^2 -1)^2},\, \frac{-3t^4 - 10t^2 -3}{4(t^2 -1)^2},\, \frac{-3t^2- 1}{2(t^2 - 1)}\bigg), \hspace{.5cm} (c_1, c_3)=\Big(-\frac{5}{16},-\frac{13}{16}\Big)\]
for some $t\in\mathbb{Q}$. However, equating expressions for $c_1$ we obtain an equation $\frac{-15t^4 - 2t^2 + 1}{4(t^2 -1)^2}=-\frac{5}{16}$ with no rational solutions, a contradiction.   
\\[7pt] 
\emph{Subcase(4.8)}: The pair of tuples $(c_1, c_2,P)$ and $(c_1,c_3)$ take the form: \vspace{.15cm}
\[(c_1, c_2,P)=\bigg(\frac{-15t^4 - 2t^2 + 1}{4(t^2 -1)^2},\, \frac{-3t^4 - 10t^2 -3}{4(t^2 -1)^2},\, \frac{-3t^2- 1}{2(t^2 - 1)}\bigg), \hspace{.5cm} (c_1, c_3)=\Big(-\frac{21}{16},-\frac{13}{16}\Big)\]
for some $t\in\mathbb{Q}$. However, equating expressions for $c_1$ we obtain an equation $\frac{-15t^4 - 2t^2 + 1}{4(t^2 -1)^2}=-\frac{21}{16}$ with no rational solutions, a contradiction. 
\\[7pt]
\emph{Subcase(4.9)}: The pair of tuples $(c_1, c_2)$ and $(c_1,c_3,P)$ take the form: \vspace{.15cm}
\[(c_1, c_2)=\Big(-\frac{5}{16},-\frac{13}{16}\Big), \hspace{.5cm} (c_1,c_3,P)=\bigg(\frac{1-y^2}{4},\,\frac{-3-y^2}{4},\, \frac{1+y}{2}\bigg),\]
for some $y\in\mathbb{Q}$. However, by interchanging $c_2$ and $c_3$, we may proceed as in Subcase 4.3 above.  
\\[7pt] 
\emph{Subcase(4.10)}: The pair of tuples $(c_1, c_2)$ and $(c_1,c_3,P)$ take the form: \vspace{.15cm}
\[(c_1, c_2)=\Big(-\frac{5}{16},-\frac{13}{16}\Big), \hspace{.5cm} (c_1,c_3,P)=\bigg(\frac{-15t^4 - 2t^2 + 1}{4(t^2 -1)^2},\, \frac{-3t^4 - 10t^2 -3}{4(t^2 -1)^2},\, \frac{-3t^2- 1}{2(t^2 - 1)}\bigg),\]
for some $t\in\mathbb{Q}$. However, by interchanging $c_2$ and $c_3$, we may proceed as in Subcase 4.7 above.  
\\[7pt] 
\emph{Subcase(4.11)}: The pair of tuples $(c_1, c_2)$ and $(c_1,c_3)$ take the form: \vspace{.15cm}
\[(c_1, c_2)=\Big(-\frac{5}{16},-\frac{13}{16}\Big), \hspace{.5cm} (c_1, c_3)=\Big(-\frac{5}{16},-\frac{13}{16}\Big).\]
However, in this case $c_2=c_3$, a contradiction. 
\\[7pt] 
\emph{Subcase(4.12)}: The pair of tuples $(c_1, c_2)$ and $(c_1,c_3)$ take the form: \vspace{.15cm}
\[(c_1, c_2)=\Big(-\frac{5}{16},-\frac{13}{16}\Big), \hspace{.5cm} (c_1, c_3)=\Big(-\frac{21}{16},-\frac{13}{16}\Big).\]
However, equating expressions for $c_1$, we obtain an immediate contradiction. 
\\[7pt]
\emph{Subcase(4.13)}: The pair of tuples $(c_1, c_2)$ and $(c_1,c_3,P)$ take the form: \vspace{.15cm}
\[(c_1, c_2)=\Big(-\frac{21}{16},-\frac{13}{16}\Big), \hspace{.5cm} (c_1,c_3,P)=\bigg(\frac{1-y^2}{4},\,\frac{-3-y^2}{4},\, \frac{1+y}{2}\bigg),\]
for some $y\in\mathbb{Q}$. However, by interchanging $c_2$ and $c_3$, we may proceed as in Subcase 4.4 above. 
\\[7pt]
\emph{Subcase(4.14)}: The pair of tuples $(c_1, c_2)$ and $(c_1,c_3,P)$ take the form: \vspace{.15cm}
\[(c_1, c_2)=\Big(-\frac{21}{16},-\frac{13}{16}\Big), \hspace{.5cm} (c_1,c_3,P)=\bigg(\frac{-15t^4 - 2t^2 + 1}{4(t^2 -1)^2},\, \frac{-3t^4 - 10t^2 -3}{4(t^2 -1)^2},\, \frac{-3t^2- 1}{2(t^2 - 1)}\bigg),\]
for some $t\in\mathbb{Q}$. However, by interchanging $c_2$ and $c_3$, we may proceed as in Subcase 4.8 above. 
\\[7pt]
\emph{Subcase(4.15)}: The pair of tuples $(c_1, c_2)$ and $(c_1,c_3)$ take the form: \vspace{.15cm}
\[(c_1, c_2)=\Big(-\frac{21}{16},-\frac{13}{16}\Big), \hspace{.5cm} (c_1, c_3)=\Big(-\frac{5}{16},-\frac{13}{16}\Big).\]
However, equating expressions for $c_1$, we obtain an immediate contradiction. 
\\[7pt] 
\emph{Subcase(4.16)}: The pair of tuples $(c_1, c_2)$ and $(c_1,c_3)$ take the form: \vspace{.15cm}
\[(c_1, c_2)=\Big(-\frac{21}{16},-\frac{13}{16}\Big), \hspace{.5cm} (c_1, c_3)=\Big(-\frac{21}{16},-\frac{13}{16}\Big).\]
However, in this case $c_2=c_3$, a contradiction. 
\\[7pt] 
\textbf{Case(5):} Suppose that the orbit of $P$ enters a fixed point for only one map in $S$, say $\phi_1$, and enters a $3$-cycle for the others, $\phi_2$ and $\phi_3$. Then in particular, both $\phi_2$ and $\phi_3$ have rational points of period $3$. Moreover, $P$ has finite orbit for the set $S'=\{\phi_2,\phi_3\}$. However, Lemma \ref{lem:33} implies that this is impossible.            
\\[7pt]  
\textbf{Case(6):} Suppose that the orbit of $P$ enters a fixed point for only one map in $S$, say $\phi_1$, enters a $2$-cycle for another, say $\phi_2$, and enters a $3$-cycle for the other, $\phi_3$. Then in particular, $\phi_i$ has a rational point of period $i$ for all $1\leq i \leq 3$. Moreover, $P$ has finite orbit for both sets $S'=\{\phi_1, \phi_3\}$ and $S''=\{\phi_2,\phi_3\}$ simultaneously. Hence, Lemma \ref{lem:13} applied to $S'$ implies that $c_1=-21/16$ and $c_3=-29/16$. Likewise, Lemma \ref{lem:23} applied to $S''$ implies that $c_2=-21/16$ and $c_3=-29/16$. Therefore, we deduce that $c_1=c_2$, a contradiction.    
\\[7pt]
\textbf{Case(7):} Suppose that the orbit of $P$ enters a $2$-cycle for every map in $S$. In particular, each map in $S$ has a rational point of period two, and by replacing $P$ with $\phi_1^n(P)$ for some $n$, we may assume that a point of period two for $\phi_1$, renamed $P$, has finite orbit for both sets $S'=\{\phi_1,\phi_2\}$ and $S''=\{\phi_1,\phi_3\}$ simultaneously. Then the classification Lemma \ref{lem:22} leads us to several subcases: 
\\[7pt]
\emph{Subcase(7.1)} The pair of tuples $(c_1, c_2,P)$ and $(c_1,c_3,P)$ take the form: \vspace{.15cm}
\begin{align*}
(c_1,c_2,P)&=\bigg(\frac{-7t^4 - 2t^2 -7}{4(t^2 -1)^2},\, \frac{-3t^4 - 10t^2 -3}{4(t^2 -1)^2},\, \frac{-3t^2- 1}{2(t^2 - 1)}\bigg), \\[4pt]  
(c_1,c_3,P)&=\bigg(\frac{-7u^4 - 2u^2 -7}{4(u^2 -1)^2},\, \frac{-3u^4 - 10u^2 -3}{4(u^2 -1)^2},\, \frac{-3u^2- 1}{2(u^2 - 1)}\bigg), 
\end{align*}
for some $t,u\in\mathbb{Q}$. Equating expressions for $P$, we see that $t=\pm{u}$. However, this relationship forces $c_2=c_3$, a contradiction.  
 \\[7pt]
\emph{Subcase(7.2)} The pair of tuples $(c_1, c_2,P)$ and $(c_1,c_3)$ take the form: \vspace{.15cm}
\begin{align*}
(c_1,c_2,P)&=\bigg(\frac{-7t^4 - 2t^2 -7}{4(t^2 -1)^2},\, \frac{-3t^4 - 10t^2 -3}{4(t^2 -1)^2},\, \frac{-3t^2- 1}{2(t^2 - 1)}\bigg),\\[5pt] 
(c_1,c_3)&\in\bigg\{\Big(-\frac{3}{4},-\frac{7}{4}\Big),\Big(-\frac{7}{4},-\frac{3}{4}\Big)\;\Big(-\frac{13}{16},-\frac{21}{16}\Big),\Big(-\frac{21}{16},-\frac{13}{16}\Big) \; \Big(-\frac{37}{16},-\frac{21}{16}\Big) \bigg\},
\end{align*}
for some $t\in\mathbb{Q}$. However, equating expressions for $c_1$ we obtain equations 
\[\frac{-7t^4 - 2t^2 -7}{4(t^2 -1)^2}=-\frac{3}{4}, -\frac{13}{16}, -\frac{21}{16}\] 
with no rational solutions, a contradiction. On the other hand, if $\frac{-7t^4 - 2t^2 -7}{4(t^2 -1)^2}=-\frac{7}{4}$, then $t=0$. Hence, $(c_1,c_2,P)=(-7/4,-3/4,-1/2)$ and $(c_1,c_3)=(-7/4,-3/4)$. In particular, we see that $c_2=c_3$, a contradiction.    
\\[7pt]
\emph{Subcase(7.3)} The pair of tuples $(c_1, c_2)$ and $(c_1,c_3,P)$ take the form: \vspace{.15cm}
\begin{align*}
(c_1,c_2)&\in\bigg\{\Big(-\frac{3}{4},-\frac{7}{4}\Big),\Big(-\frac{7}{4},-\frac{3}{4}\Big)\;\Big(-\frac{13}{16},-\frac{21}{16}\Big),\Big(-\frac{21}{16},-\frac{13}{16}\Big) \; \Big(-\frac{37}{16},-\frac{21}{16}\Big) \bigg\}, \\[5pt] 
(c_1,c_3,P)&=\bigg(\frac{-7t^4 - 2t^2 -7}{4(t^2 -1)^2},\, \frac{-3t^4 - 10t^2 -3}{4(t^2 -1)^2},\, \frac{-3t^2- 1}{2(t^2 - 1)}\bigg) 
\end{align*}
for some $t\in\mathbb{Q}$. Interchanging $(c_1,c_2)$ with $(c_2,c_3)$, we may proceed as in Subcase 7.2. 
\\[7pt]
\emph{Subcase(7.4)} The pair of tuples $(c_1, c_2)$ and $(c_1,c_3,P)$ take the form: \vspace{.15cm}
\begin{align*}
(c_1,c_2)&\in\bigg\{\Big(-\frac{3}{4},-\frac{7}{4}\Big),\Big(-\frac{7}{4},-\frac{3}{4}\Big)\;\Big(-\frac{13}{16},-\frac{21}{16}\Big),\Big(-\frac{21}{16},-\frac{13}{16}\Big) \; \Big(-\frac{37}{16},-\frac{21}{16}\Big) \bigg\}, \\[5pt] 
(c_1,c_3)&\in\bigg\{\Big(-\frac{3}{4},-\frac{7}{4}\Big),\Big(-\frac{7}{4},-\frac{3}{4}\Big)\;\Big(-\frac{13}{16},-\frac{21}{16}\Big),\Big(-\frac{21}{16},-\frac{13}{16}\Big) \; \Big(-\frac{37}{16},-\frac{21}{16}\Big) \bigg\},
\end{align*}
However, equating expressions for $c_1$, we need only consider the cases 
\[\Scale[.95]{(c_1,c_2,c_3)\in\Big\{(-\frac{3}{4},-\frac{7}{4},-\frac{7}{4}),\;(-\frac{7}{4},-\frac{3}{4},-\frac{3}{4})\;(-\frac{13}{16},-\frac{21}{16},-\frac{21}{16}),\;(-\frac{21}{16},-\frac{13}{16},-\frac{13}{16}),\;(-\frac{37}{16},-\frac{21}{16}, -\frac{21}{16}) \Big\}}.\] 
However, in all cases, $c_2=c_3$, a contradiction of our assumption that the $c_i$'s are distinct. 
\\[7pt] 
\textbf{Case(8):} Suppose that the orbit of $P$ enters a $3$-cycle for every map in $S$. Then in particular, both $\phi_1$ and $\phi_2$ have rational points of period $3$. Moreover, $P$ has finite orbit for the set $S'=\{\phi_1,\phi_2\}$. However, Lemma \ref{lem:33} implies that this is impossible.            
 \\[7pt]
\textbf{Case(9):} Suppose that the orbit of $P$ enters a $2$-cycle for two maps in $S$, say $\phi_1$ and $\phi_2$, and enters a $3$-cycle for the third $\phi_3$. Then in particular, $\phi_1$ and $\phi_2$ have rational points of period $2$, and $\phi_3$ has a rational point of period $3$. Moreover, $P$ has finite orbit for both sets $S'=\{\phi_1, \phi_3\}$ and $S''=\{\phi_2,\phi_3\}$ simultaneously. Hence, Lemma \ref{lem:23} applied to $S'$ implies that $c_1=-21/16$ and $c_3=-29/16$. Likewise, Lemma \ref{lem:23} applied to $S''$ implies that $c_2=-21/16$ and $c_3=-29/16$. Therefore, $c_1=c_2$, a contradiction.      
\\[7pt] 
\textbf{Case(10):} Suppose that the orbit of $P$ enters a $2$-cycle for one map in $S$, say $\phi_1$, and enters a $3$-cycle for the others, $\phi_2$ and $\phi_3$. Then  in particular, both $\phi_2$ and $\phi_3$ have rational points of period $3$. Moreover, $P$ has finite orbit for the set $S'=\{\phi_2,\phi_3\}$. However, Lemma \ref{lem:33} implies that this is impossible.            
\end{proof}
In particular, Case $(1)$ through Case $(10)$ above imply that if $S=\{x^2+c_1,x^2+c_2, x_2+c_3\}$ has a finite orbit point $P\in\mathbb{Q}$, then 
\begin{equation}\label{classification:3maps}
\begin{split} 
S&=\bigg\{x^2-\frac{5}{16},\, x^2-\frac{13}{16},\, x^2-\frac{21}{16}\bigg\}\;\;\text{and}\;\;P\in\bigg\{\pm\frac{1}{4},\,\pm\frac{3}{4},\, \pm\frac{5}{4}\bigg\},\\[2pt] 
\text{or}\;\;\;\;\;\;\;&\\[2pt] 
S&=\bigg\{x^2+\frac{3}{16},\, x^2-\frac{5}{16},\, x^2-\frac{13}{16}\bigg\}\;\;\text{and}\;\;P\in\bigg\{\pm\frac{1}{4},\,\pm\frac{3}{4}\bigg\}. 
\end{split} 
\end{equation} 
This completes the proof of statement (1) of Theorem \ref{thm:classification}. As for statement (2), suppose that $S$ has at least four quadratic polynomials, $\phi_i$ for $1\leq i\leq4$, and that $P\in\mathbb{Q}$ has finite orbit for $S$. Then $P$ has finite orbit for the subsets $S'=\{\phi_1,\phi_2,\phi_3\}$ and $S''=\{\phi_1,\phi_2,\phi_4\}$ simultaneously. Therefore, statement (1) of Theorem \ref{thm:classification} implies that $(S',P)$ and $(S'',P)$ must correspond to one the pairs in (\ref{classification:3maps}) above. In particular, since $S'\neq S''$, it must be the case that 
\[S'\cup S''=\bigg\{x^2+\frac{3}{16},x^2-\frac{5}{16},\, x^2-\frac{13}{16},\, x^2-\frac{21}{16}\bigg\}\;\;\;\text{and}\;\; P\in\bigg\{\pm\frac{1}{4},\,\pm\frac{3}{4}\bigg\}.\] 
However, $P$ must have finite orbit for $S'\cup S''$, and this is not the case: in all cases for $P$, the point $Q:=(f_4\circ f_2\circ f_1\circ f_4)(P)$ satisfies $f_1^4(Q)\neq f_1^2(Q)$. Therefore, Theorems 1, 2, and 3 of \cite{Poonen} imply that $Q$ is not preperiodic for $f_1$. Hence, $P$ cannot be a finite orbit point for $S$, a contradiction. 
 

\begin{thebibliography}{14}
\bibitem{Baker-DeMarco} M. Baker and L. DeMarco, Preperiodic points and unlikely intersections, \emph{Duke Math. Journal} 159.1 (2011): 1-29.
\bibitem{monoid1} D.R. Heath-Brown, and G. Micheli, Irreducible polynomials over finite fields produced by composition of quadratics,  by composition of quadratics, \emph{Revista Matem\'{a}tica}, to appear.  
\bibitem{monoid2} A. Ferraguti, G. Micheli, and Reto Schnyder, Irreducible compositions of degree two polynomials over finite fields have regular structure, \emph{Quart. J. of Math.} 69.3 (2018): 1089-1099.
\bibitem{5cycle} E.V. Flynn, B. Poonen, and E. F. Schaefer, Cycles of quadratic polynomials and rational points on a genus-two curve, \emph{Duke Math. J.} 90 (1997): 435-463.
\bibitem{VOHWH} V. O. Healey and Wade Hindes, Stochastic Canonical heights, submitted, arXiv:1805.10897.
\bibitem{Hutz-Ingram} B. Hutz and P. Ingram, On Poonen's conjecture concerning rational preperiodic points of quadratic maps, \emph{Rocky Mountain Journal of Mathematics} 43.1 (2013): 193-204. 
\bibitem{Kawaguchi} S. Kawaguchi, Canonical heights for random iterations in certain varieties, \emph{Int. Math. Res. Not.}, Article ID rnm023, 2007.
\bibitem{4cycle} P. Morton, Arithmetic properties of periodic points of quadratic maps, \emph{Acta Arithmetica} 62.4 (1992): 343-372.
\bibitem{Morton-Silverman} P. Morton and J. Silverman, Rational periodic points of rational functions, \emph{Int. Math. Res. Not.} (IMRN), 1994.2 (1994): 97-110.
\bibitem{monoid3} A. Ostafe and Marley Young, On algebraic integers of bounded house and preperiodicity in polynomial semigroup dynamics, preprint arXiv:1807.11645.
\bibitem{Poonen} B. Poonen, The complete classification of rational preperiodic points of quadratic polynomials over $\mathbb{Q}$: a refined conjecture, \emph{Math. Z.}, 228.1 (1998): 11-29. 
\bibitem{SilvDyn} J. Silverman, The Arithmetic of Dynamical Systems, Vol. 241, Springer GTM, 2007. 
\bibitem{6cycle} M. Stoll, Rational $6$-cycles under iteration of quadratic polynomials, \emph{LMS Journal of Computation and Mathematics}, 11 (2008): 367-380.
\end{thebibliography}
\end{document}